\documentclass[sn-mathphys]{sn-jnl}

\jyear{2021}%

\theoremstyle{thmstyleone}%
\newtheorem{theorem}{Theorem}
\theoremstyle{thmstyletwo}%
\newtheorem{example}{Example}%
\newtheorem{remark}{Remark}%
\newtheorem{lemma}{Lemma}
\theoremstyle{thmstylethree}%

\usepackage{lipsum}
\usepackage{amsmath}
\usepackage{amsfonts}
\usepackage{bm}
\usepackage{bbm}
\usepackage{subfigure}
\usepackage{graphicx}
\usepackage{amssymb}
\usepackage{enumerate}
\usepackage{caption}
\usepackage{epstopdf}
\usepackage{latexsym,amscd}
\usepackage{amssymb,amsthm}
\usepackage{epsfig}
\usepackage{mathrsfs}

\ifpdf
  \DeclareGraphicsExtensions{.eps,.pdf,.png,.jpg}
\else
  \DeclareGraphicsExtensions{.eps}
\fi

\renewcommand{|}{\vert}

\begin{document}

\title[A novel class of numerical schemes for FDEs]{A novel class of arbitrary high-order numerical schemes for fractional differential equations}


\author[1]{\fnm{Peng} \sur{Ding}}\email{dingpeng@stu.xmu.edu.cn}

\author*[2]{\fnm{Zhiping} \sur{Mao}}\email{zmao@eitech.edu.cn}


\affil[1]{\orgdiv{School of Mathematical Sciences, Fujian Provincial Key Laboratory of Mathematical Modeling and High-Performance Scientific Computing}, \orgname{Xiamen University}, \orgaddress{\city{Xiamen}, \postcode{361005}, \state{Fujian}, \country{China}}}

\affil*[2]{\orgdiv{School of Mathematical Sciences}, 
\orgname{Eastern Institute of Technology}, \orgaddress{\city{Ningbo}, \postcode{315200}, \state{Zhejiang}, \country{China}}}



\abstract{It is well-known that designing efficient, high-order, and stable numerical schemes for time non-local or fractional differential equations faces three key challenges: non-locality, low solution regularity, and long-term simulation. Achieving this while minimizing storage costs is particularly difficult, especially when attempting to address all three issues \emph{simultaneously}. 
In this work, we propose a novel class of numerical schemes designed to \emph{simultaneously} address the three key challenges associated with time fractional differential equations (TFDEs). 
To achieve this, we derive an equivalent integer-order \emph{extended parametric differential equation} (EPDE) by dimensional expanding for the TFDE and establish its corresponding stability.
Remarkably, for the resulting EPDE, we provide a rigorous analysis demonstrating that it exhibits high regularity with respect to the extended parameter dimension. This finding motivates us to apply a spectral method for discretizing the extended parametric space, enabling high accuracy.
Consequently, we employ the Jacobi spectral collocation method combined with characteristic decomposition, resulting in $M$ independent integer-order ordinary differential equations (ODEs), where $M$ represents the number of nodes used for the spectral collocation method.
Therefore, it is straightforward to apply any traditional numerical scheme for the resulting $M$ independent integer ODEs.  
In the present work, we utilize the BDF-$k$ ($k=1,\ldots ,5$) formulas for the time discretization and conduct a rigorous stability analysis. Additionally, we provide an error estimate for the fully discretization scheme, demonstrating that the convergence order is $O(\Delta t^{k} + M^{-m})$, where $\Delta t$ represents the time step size. 
More importantly, both theoretical and numerical results reveal that the spectral approximation in the extended parametric space exhibits exponential convergence. This implies that a \emph{small fixed} value of $M$ is sufficient to achieve high accuracy in the $\theta$-direction. In fact, numerical results demonstrate that choosing $ M=30 $ is sufficient to attain an accuracy of approximately $O(10^{-12})$. 
Since $M$ is fixed, we conclude that the computational cost and storage requirements of our proposed algorithm are essentially the same as those for ODEs, with a computational cost of $O(N)$ and a storage requirement of $O(1)$, where $N$ is the total number of time steps. 
We present several numerical examples, including both linear and nonlinear problems, to demonstrate the effectiveness of the proposed method and to validate the theoretical results.}




\maketitle

\section{Introduction}
Fractional/non-local differential equations have been effective in modeling anomalous diffusion as well as capturing long-range spatio-temporal interactions \cite{Hairer2003,Mao2019,QiangDu,QiangDu2}. Please refer to \cite{Meerschaert_2014,Can,HANYGA20011399} for more applications in poroelasticity, ground water hydrology, and geophysical flows. Analytic solutions of fractional differential equations  are typically obtained by using special functions for simple model problems. However, it is difficult to obtain closed forms of solutions for complex fractional differential equations. Hence, numerically solving fractional differential equations is turning out to be a popular topic. 

We consider in this work the following time fractional differential equations (TFDEs) 
\begin{equation}\label{fractional ODE}
\begin{aligned}
     \frac{d ^{\alpha}\phi}{d t^{\alpha}} = F(t, \phi(t)),\quad t>0,\quad 
      \phi(0) = \phi_{0},
\end{aligned}
\end{equation}
where $0<\alpha <1$ and the fractional operator is the so-called Caputo fractional derivative defined by 
\begin{equation*}\label{fractional derivative}
    \frac{d^{\alpha}\phi}{dt^{\alpha}} = {_{0}}I_{t}^{1-\alpha}\frac{d\phi(t)}{dt} = \frac{1}{\Gamma(1-\alpha)}\int_{0}^{t}\frac{1}{(t-\tau)^{\alpha}}\frac{d\phi}{d\tau}d\tau,
\end{equation*}
here $_{0}I_{t}^{\alpha}$ is fractional integral 
\begin{equation*}\label{eq:galpha}
    _{0}I_{t}^{\alpha} \phi(t) = g_{\alpha}(t)*\phi(t),\,\quad g_{\alpha}(t) = t^{\alpha -1}/\Gamma(\alpha),
\end{equation*}
 the operator $*$ denotes the convolution, $\Gamma(\cdot)$ is the Gamma function, $F(t, \phi(t))$ is continuous, bounded, and fulfills the Lipschitz condition with respect to the second variable such that TFDEs (\ref{fractional ODE}) is well-posed (see \cite{Diethelm2002}, Theorem 2.1 and 2.2).
 
 To numerically solve TFDEs (\ref{fractional ODE}), there are three main issues in designing efficient and high-accuracy schemes:
\begin{itemize}
    \item \emph{Singularity}. The singular kernel gives rise to singular solutions of TFDEs leading to difficulty in designing high-order schemes.
    \item \emph{Non-locality}. Another issue arising from the kernel is the non-locality leading to high computational complexity and storage.
    \item \emph{Long time simulation}. The third challenge lies in the difficulty of designing efficient algorithms for long-time simulations of non-local problems.
\end{itemize}

To address the above challenges, extensive efforts have been undertaken. 
Advanced high-order numerical schemes, including those based on the finite difference method and spectral method, have been devised specifically for handling smooth solutions \cite{Lin2007,Sun2006,Li2009,Gao2014,ALIKHANOV2015424}. However, the solutions to TFDEs typically exhibit lower regularity since the singular kernel of fractional operator, which  challenges in the design of high-accuracy numerical schemes. 
To resolve this issue, Jin et al.~\cite{Jin2017} employed the correction technique to ensure the optimal convergence order corresponding to the proposed numerical schemes for non-smooth solution. More literature based on correction techniques can be referenced in \cite{Jin2016,Jin2017,Yan2018,Shi2023}. 
Stynes et al.\cite{Martin2017,Kassem2020,KoptevaMeng} utilized non-uniform time grids (such as graded meshes and geometric meshes) to capture the weak singularity at $t=0$ under appropriate smoothness of the solution, thereby achieving optimal convergence orders for the numerical schemes under suitable parameter choices.  The aforementioned two methods either involve a high complexity in designing correction terms for TFDEs or find it difficult to design high-order numerical schemes for TFDEs with non-smooth solutions.
Spectral methods employing specially designed basis functions have been developed to address the singularity issue; however, these approaches are often restricted to simple model problems or face challenges with long-time simulations ~\cite{ChengShenWang, XuHou_Muntz, XuCui_Muntz}. 

On the other hand, enormous numerical schemes of TFDEs based on original L1 scheme \cite{Lin2007,Sun2006}, L2-1$_{\sigma}$ scheme \cite{ALIKHANOV2015424}, L2 scheme \cite{Gao2014,Zhu} invariably suffer from high computational complexity and storage costs. This
is due to the non-locality of the fractional derivatives. To overcome this issue, fractional order operators can be expressed  sum of exponential by using numerical approximations of the convolution kernel  function $g_{\alpha}(t)$  such that to design the efficient algorithm for TFDEs. For instance, Lubich \cite{Lubich2002} first expressed the convolution kernel $g_{\alpha}(t)$ in the form of its inverse Laplace transform $g_{\alpha}(t) = \frac{1}{2\pi i}\int_{\mathcal{H}}\mathcal{L}[g_{\alpha}]e^{t\lambda}d\lambda$, where $\mathcal{L}[g_{\alpha}](\lambda)$ denotes the Laplace transform of $g_{\alpha}(t)$ and $\mathcal{H}$ is a suitable contour. Subsequently, one can directly discretize the inverse Laplace transform by the appropriate numerical integration formulas to obtain efficient numerical schemes for fractional order operators. The storage and computational cost of the fast methods in \cite{Lubich2002} are $O(\log N)$ and $O(N\log N)$, which significantly less than  the direct methods with $O(N)$ memory and $O(N^{2})$ operations where $N$ is the total number of time steps. This approach can be applied to construct efficient methods for a wide class of non-local models. Some other relate works include \cite{Lubich2002, Ban2017, Lubich2006, Hesthaven2017jsc, Hesthaven2017,tensorFEM}. 

Another approach to obtaining the efficient algorithms is based on the following formula
\begin{equation}\label{kernel}
    g_{\alpha}(t) = \frac{sin(\alpha\pi)}{\pi}\int_{0}^{\infty}s^{-\alpha}e^{-ts}ds,\quad 0<\alpha< 1.
\end{equation}
Jiang et al.~\cite{Jiang2017} obtained a finite integral by truncating the above infinite integral, then combined Jacobi-Gauss quadrature and multi-domain Legendre-Gauss quadrature to discretize the finite integral, thereby they obtained efficient numerical algorithms for fractional order operators with  the storage
cost and computational complexity are $O(N_{\varepsilon})$ and $O(NN_{\varepsilon})$, respectively, where $N_{\varepsilon}$ is the number of sum-of-exponential truncated terms. Interested readers may refer to  \cite{Yan2017}. Zeng \cite{Zeng2018} approximated  kernel (\ref{kernel}) by truncated Laguerre-Gauss quadrature, and the discretization error caused from the Laguerre-Gauss quadrature is independent of the step size. Li \cite{Li2010} transformed kernel (\ref{kernel}) into its equivalent form, then multi-domain Legendre–Gauss quadrature was applied to approximate the transformed integral to obtain the efficient algorithm. Some other related works includes Guo \cite{Guo2019}, Chen \cite{Chen2022}, etc. The aforementioned methods all employ suitable quadrature techniques to discretize the convolution kernel $g_{\alpha}(t)$ directly, leading to the sum-of-exponentials form for accelerated computation of TFDEs. The interested readers can refer to \cite{Lubich2006,Li2010,Beylkin2010,Chen2022}.
However, such methods cannot simultaneously account for the singularity of the solution at the initial time, making it difficult to design high-order numerical schemes.

In summary, the aforementioned work cannot simultaneously solve the three issues. The aim of this work is to design efficient, high-order and stable schemes at a low cost of storage for TFDEs (\ref{fractional ODE}), capable of enabling long-time simulations. Specifically, inspired by the work in \cite{Fritz2021}, we derive an equivalent integer-order \emph{extended  parametric differential equations} (EPDE) for the TFDEs by dimensional expanding. Building on this framework, we design efficient and high-order numerical schemes for the EPDE. This approach provides a fresh perspective for developing high-order and efficient methods for TFDEs.
The main contributions of this work are given as follows:
\begin{itemize}
\item  We derive equivalent {integer-order parametric differential equations} for the TFDEs and give rigorous analysis on the stability as well as the regularity with respect to extended parametric $\theta$-direction.
\item We turn the original FDE into $M$ independent integer-order ODEs, where $M$ is the number of nodes used for the spectral collocation method. In particular, 
we use the Jacobi spectral collocation method for the $\theta$-direction and apply the characteristic decomposition generating $M$ independent integer-order ODEs. Then, we apply BDF-$k$ formulae for the time discretization. We show the stability and give a rigorous error estimate with order $O(\Delta t^{k} + M^{-m})$ for the proposed method, where $\Delta t$ is the time step size.
\item Both theoretical and numerical results indicate that we have spectral accuracy for the spectral approximation, which means that only a small value of $M$ is enough to guarantee a high accuracy for the $\theta$-direction. Therefore, we can fix a small value of $M$ for the spectral discretization so that the computational cost and storage requirement of our proposed algorithm are $O(N)$ and $O(1)$, respectively.
\item  We present the stable region and further demonstrate the effectiveness of the proposed method by linear and non-linear examples.
\end{itemize}

The rest of the paper is structured as follows. In Section \ref{sec:2}, we show the equivalence between TFDE and EPDE, and establish the stability for the EPDE. We present the numerical schemes and establish the stability as well as the error estimate in Section \ref{sec:3}. Several numerical examples are given in Section \ref{sec:4}, and a concluding remark is given in Section \ref{sec:5}. 

\section{Equivalent EPDE and its stability}\label{sec:2}

In this section, we begin by deriving an equivalent integer-order extended parametric differential equation for the TFDEs \eqref{fractional ODE}, and establish its stability.

\subsection{A equivalent EPDE}
Now let us derive an equivalent integer differential equation for the fractional problem \eqref{fractional ODE}. 
To begin with, we apply the fractional integral on both sides of \eqref{fractional ODE} to get the solution of the fractional problem (see also \cite{Diethelm2002}, Lemma 2.1):
\begin{equation}\label{fractional integral}
    \phi(t) = \phi_{0} + \int_{0}^{t}\frac{(t-\tau)^{\alpha -1}}{\Gamma(\alpha)}F(\tau,\phi(\tau))d\tau = \phi_{0}+ \int_{0}^{t}g_{\alpha}(t-\tau)F(\tau,\phi(\tau))d\tau,
\end{equation}
where $g_{\alpha}(t) = t^{\alpha -1}/\Gamma(\alpha)$. By the definition of Gamma function (see \cite{Diethelm2002}, Definition 1.2), i.e., 
$\Gamma(1-\alpha)t^{\alpha -1} = \int_{0}^{\infty} m^{-\alpha}e^{-mt}dm,$ 
we deduce 
\begin{equation*}
\label{relationship}
    \begin{aligned}
    g_{\alpha}(t)& = \frac{t^{\alpha -1}}{\Gamma(\alpha)} 
    = \frac{\Gamma(1-\alpha)t^{\alpha -1}}{\Gamma(1-\alpha)\Gamma(\alpha)}
    = \frac{1}{\Gamma(1-\alpha)\Gamma(\alpha)}\int_{0}^{\infty}m^{-\alpha}e^{-mt}dm\\
        &=\frac{1}{\Gamma(1-\alpha)\Gamma(\alpha)}\int_{0}^{1}c_{1}(\theta)^{-\alpha}e^{-c_{1}(\theta)t}c_{0}(\theta)^{2}d\theta
        =\int_{0}^{1}c_{0}(\theta)e^{-c_{1}(\theta)t}\omega_{\alpha}(\theta)d\theta,
    \end{aligned}
\end{equation*}
where 
\begin{equation}\label{definition of omega}
    c_{0}(\theta) = \frac{1}{1-\theta},\quad 
    c_{1}(\theta) = \frac{\theta}{1-\theta},\quad  \omega_{\alpha}(\theta) = \frac{\theta^{-\alpha}(1-\theta)^{\alpha-1}}{\Gamma(1-\alpha)\Gamma(\alpha)},
\end{equation}
and in the derivation, we use the change of variable $m = c_{1}(\theta)$ and the equalities $\Gamma(1-\alpha)\Gamma(\alpha)\omega_{\alpha}(\theta) = c_{0}(\theta)c_{1}(\theta)^{-\alpha}$ and $c_{1}(\theta)' = (c_{0}(\theta))^2$.
Substituting the above two equations into \eqref{fractional integral} gives 
\begin{equation*}\label{r2}
\begin{aligned}
    \phi(t) 
   & = \phi_{0} +\int_{0}^{t}\int_{0}^{1}c_{0}(\theta)e^{-c_{1}(\theta)(t-\tau)}\omega_{\alpha}(\theta)d\theta F(\tau,\phi(\tau))d\tau\\
    &=\int_{0}^{1}\bigg[\phi_{0}  + c_{0}(\theta)\int_{0}^{t}e^{-c_{1}(\theta)(t-\tau)}F(\tau,\phi(\tau))d\tau\bigg]\omega_{\alpha}(\theta)d\theta.
\end{aligned}
\end{equation*}
Here we use the equality $\int_{0}^{1}\omega_{\alpha}(\theta)d\theta = 1$.

Let us define the function 
\begin{equation*}
    \varphi(t,\theta) := \phi_{0} + c_{0}(\theta)\int_{0}^{t}e^{-c_{1}(\theta)(t-\tau)}F(\tau,\phi(\tau))d\tau.
\end{equation*}
Consequently, we have from the above two equations that the solution of the TFDEs \eqref{fractional ODE} is given by 
\begin{equation}\label{eq:transform}
    \phi(t) = \int_{0}^{1} \varphi(t,\theta)\omega_{\alpha}(\theta)d\theta := \mathcal{C}[\varphi](t).
\end{equation}
Observe that the function $\varphi(t,\theta)$ is the solution of the following parametric integer-order ordinary differential equation:
\begin{equation}\label{equation}
\begin{aligned}
    &\frac{\partial \varphi(t,\theta)}{\partial t} + c_{1}(\theta)\varphi(t,\theta) = c_{0}(\theta)F(t,\mathcal{C}[\varphi](t)) + c_{1}(\theta)\phi_{0},\quad\theta\in (0,1),t>0,\\
    &\varphi(0,\theta) = \phi_{0},\quad\theta\in (0,1).
\end{aligned}
\end{equation}
We refer to the above equation as the  \emph{Extended Parametric Differential Equation} (EPDE), which is an \emph{integer-order} parametric equation. Once we solve the above integer equation, the solution of the corresponding TFDEs can be computed immediately using \eqref{eq:transform}.
Assume that $F(\cdot,\cdot)$ satisfies the Lipschitz condition with respect to the second variable, then the above parametric equation is well-posed with the help of Picard theorem (see \cite[Page 62]{odinarydifferentialequation}).

\subsection{Stability}
Before designing the numerical schemes, we establish the stability of EPDE \eqref{equation}.
In particular, 
define 
\begin{equation}\label{eq:alp01}
    \omega_{\alpha,0}(\theta) = (1-\theta)\omega_{\alpha}(\theta),\quad 
    \omega_{\alpha,1}(\theta) = \theta\omega_{\alpha}(\theta),
\end{equation}
where $\omega_{\alpha}$ is given in \eqref{definition of omega},
we have the following Theorem.
\begin{theorem}\label{thm:stab:cont}
 For $0<t\leq T$, $\Omega = (0,1)$, if $F(t,\phi)\phi \leq L|\phi|$, where $L$ is a positive constant, then there exists a constant $\epsilon>0$ such that the following estimate holds 
\begin{equation}\label{eq:sta:cont}
\begin{aligned}
    \|\varphi\|^2_{L^{\infty}\left(0,T;L^2_{\omega_{\alpha,0}}(\Omega)\right)} 
    + \|\varphi\|^2_{L^{2}\left(0,T;L^2_{\omega_{\alpha}}(\Omega)\right)}
\leq C(\epsilon,L, \alpha, \phi_0, T),
\end{aligned}
\end{equation}
where 
$$C(\epsilon,L, \alpha, \phi_0, T) =2\left(\phi_0^{2} \alpha e^{2\epsilon T} +
    \left(\frac{L^2}{4\epsilon}\alpha + (L^2 +\phi_0^{2})(1-\alpha)\right) e^{2\epsilon T}T\right).$$
\end{theorem}
\begin{proof}
    Taking the $L^{2}$ inner product for the first equation in (\ref{equation}) with  $\omega_{\alpha,0}(\theta)\varphi(t,\theta)$, we get
\begin{equation*}
    \int_{0}^{1}\varphi\frac{\partial \varphi}{\partial t}\omega_{\alpha,0} d\theta + \int_{0}^{1}\varphi^{2}\omega_{\alpha,1}d\theta = F(t,\mathcal{C}[\varphi])\int_{0}^{1}\varphi\omega_{\alpha}d\theta + \int_{0}^{1}\phi_{0}\varphi\omega_{\alpha,1}d\theta,
\end{equation*}
where $\omega_{\alpha,1}(\theta) = \theta\omega_{\alpha}(\theta)$ given in \eqref{eq:alp01}. Note that $\phi = \mathcal{C}[\varphi]=\int_{0}^{1}\varphi(t,\theta)\omega_{\alpha}(\theta)d\theta$. 
Thus, we have 
\begin{equation}\label{stability:01}
\begin{aligned}
    \frac{1}{2}\frac{d}{dt}\int_{0}^{1}\varphi^{2}\omega_{\alpha,0}d\theta + \int_{0}^{1}\varphi^{2}\omega_{\alpha,1}d\theta 
    &= F(t,\phi)\phi + \int_{0}^{1}\phi_{0}\varphi\omega_{\alpha,1}d\theta \\
    & \leq L|\phi| + \int_{0}^{1}\phi_{0}^{2}\omega_{\alpha,1}d\theta + \frac{1}{4}\int_{0}^{1}\varphi^{2}\omega_{\alpha,1}d\theta.
\end{aligned}
\end{equation}
We then bound $L|\phi|$ by using Cauchy-Schwarz and Young inequalities as follows:
\begin{equation}\label{stability:02}
\begin{aligned}
    L|\phi| 
    &= L\left|\int_{0}^{1}\varphi\omega_{\alpha}d\theta \right|
    =  L\left|\int_{0}^{1}\varphi(\omega_{\alpha,0} + \omega_{\alpha,1})d\theta \right|
    \leq L\int_{0}^{1}|\varphi|(\omega_{\alpha,0} + \omega_{\alpha,1})d\theta
   \\
   &\leq 
   \frac{L^2}{4\epsilon}\int_{0}^{1}\omega_{\alpha,0}d\theta + \epsilon\int_{0}^{1}\varphi^{2}\omega_{\alpha,0}d\theta
   + L^2\int_{0}^{1}\omega_{\alpha,1}d\theta + \frac{1}{4}\int_{0}^{1}\varphi^{2}\omega_{\alpha,1}d\theta,
\end{aligned}
\end{equation}
where $\epsilon>0$ is a constant. 
We then deduce from \eqref{stability:01} and \eqref{stability:02} that 
\begin{equation*}
    \frac{d}{dt}\int_{0}^{1}\varphi^{2}\omega_{\alpha,0}d\theta + \int_{0}^{1}\varphi^{2}\omega_{\alpha,1}d\theta 
    \leq 2\epsilon \int_{0}^{1}\varphi^{2}\omega_{\alpha,0}d\theta
    + 2\left(\frac{L^2}{4\epsilon}\alpha + (L^2 +\phi_0^{2})(1-\alpha) \right),
\end{equation*}
where we use the equalities 
\begin{equation}\label{eq:alp}
    \int_{0}^{1}\omega_{\alpha,0}(\theta)d\theta = 
\alpha,\; \text{and } \int_{0}^{1}\omega_{\alpha,1}(\theta)d\theta = 1-\alpha.
\end{equation}
%
We then have by the Gronwall Lemma that 
\begin{equation*}
    \int_{0}^{1}\varphi^{2}\omega_{\alpha,0}d\theta + \int_{0}^{t}\int_{0}^{1}\varphi^{2}\omega_{\alpha,1}d\theta dt
    \leq \phi_0^{2} \alpha e^{2\epsilon t} +
    \left(\frac{L^2}{4\epsilon}\alpha + (L^2 +\phi_0^{2})(1-\alpha)\right) e^{2\epsilon t}t.
\end{equation*}
Thus, the estimate \eqref{eq:sta:cont} follows from $\omega_{\alpha} = \omega_{\alpha,0} + \omega_{\alpha, 1}$ and 
\begin{equation*}
    \|\varphi\|_{L^{2}\left(0,T;L^2_{\omega_{\alpha,0}}(\Omega)\right)}
    \leq
    \|\varphi\|_{L^{\infty}\left(0,T;L^2_{\omega_{\alpha,0}}(\Omega)\right)}.
\end{equation*}
\end{proof}

\section{Numerical discretization and error analysis}\label{sec:3}
We now present the numerical schemes for problem \eqref{equation} in Subsection \ref{sec:discret} and provide the corresponding error estimate in Subsection \ref{sec:errorest}.

\subsection{Numerical discretization}\label{sec:discret}
Let us first give the numerical schemes. Specifically, we use the BDF-$k$ scheme for the time discretization while we use the Jacobi spectral collocation method for the discretization of the extended $\theta$ direction. 

We point out here that we only show the present schemes for the linear case, i.e., 
$$F(t,\mathcal{C}[\varphi](t)) = -\lambda\mathcal{C}[\varphi](t) + f(t).$$ 
 For the nonlinear case, we shall use the Picard iteration method in the implementation.
Furthermore, for simplicity, we set $f(t) = 0$ in the following analysis.

\subsubsection{Temporal-discretization and its stability}
For a given integer $N$, let $\Delta t = \frac{T}{N}$ be the time step size, $t_{n} = n\Delta t, n = 0,1,\ldots,N$ be the grid points of $I = [0,T]$. 
Therefore, by multiplying $1-\theta$ on both sides 
of EPDE \eqref{equation}, we get the BDF-$k$ scheme for \eqref{equation} as follows:
\begin{equation}\label{semi-discrete}
    \frac{(1 - \theta)}{\Delta t}\mathcal{L}\varphi^{n+1}(\theta) + \theta \varphi^{n+1}(\theta) = -\lambda\mathcal{C}[\varphi^{n+1}] + \theta\phi_{0},
\end{equation}
where 
\begin{equation}\label{the definite of L}
    \mathcal{L}\varphi^{n+1} = \alpha_{k}\varphi^{n+1} - \sum\limits_{j = 0}^{k-1}b_{j}^{(k)}\varphi^{n-j},
\end{equation}
$\alpha_{k}$ and $b_{j}^{(k)},\; (k = 1,\ldots,5)$ are given in  Table \ref{BDFk}, and $\varphi^n(\theta)$ is the approximation of $\varphi(t, \theta)$ at time $t_n$.
\begin{table}[]
    \centering
   \begin{tabular}{cccccccc}
   \toprule
   $k$ & $\alpha_{k}$ & $b_{0}^{(k)}$ & $b_{1}^{(k)}$ & $b_{2}^{(k)}$ & $b_{3}^{(k)}$ &$b_{4}^{(k)}$\\
   \midrule
   1 & 1 & 1 &  &  & & \\
   2 & $\frac{3}{2}$ & 2 & -$\frac{1}{2}$ & & &  \\[0.8ex]
   3 & $\frac{11}{6}$ & $3$ & -$\frac{3}{2}$ & $\frac{1}{3}$ & & \\[0.8ex]
   4 & $\frac{25}{12}$ & $4$ & -$3$ & $\frac{4}{3}$& -$\frac{1}{4}$ &\\[0.8ex]
   5 & $\frac{137}{60}$ & 5 & -5 & $\frac{10}{3}$ & -$\frac{5}{4}$ &$\frac{1}{5}$\\[0.8ex]
   \bottomrule
\end{tabular}
    \caption{The coefficients $\alpha_{k}$ and $b_{j}^{(k)}$ for the BDF-$k$ scheme \eqref{semi-discrete}.}
    \label{BDFk}
\end{table}
Next, we would like to show the stability of scheme \eqref{semi-discrete}.
To this end, we begin by presenting the following result (see \cite[Page 1371]{Lubich_Multiplier}), which plays an important role in proving the stability and the error analysis for scheme \eqref{semi-discrete}.    
\begin{lemma}\label{lemma1}
    For $1\leq k\leq 5$, there exist $0\leq \tau_{k} <1,\; \mathcal{L}$ is defined in \eqref{the definite of L},  a positive definite symmetric matrix $G = (g_{ij}) \in \mathcal{R}^{k,k}$ and real numbers $\delta_{0},\ldots,\delta_{k}$ such that
    \begin{equation}\label{eq:lemma1}
    \begin{aligned}
         \left(\mathcal{L}\varphi^{n+1}, \varphi^{n+1} - \tau_{k}\varphi^{n}\right) 
         = \sum\limits_{i,j = 1}^{k}g_{ij}(\varphi^{n+1+i-k},\varphi^{n+1+j-k}) &\\
       - \sum\limits_{i,j = 1}^{k}g_{ij}(\varphi^{n+i-k},\varphi^{n+j-k}) + \Vert\sum\limits_{i = 0}^{k}\delta_{i}\varphi^{n+1+i-k}\Vert^{2}, &
    \end{aligned}
    \end{equation}
    where the smallest possible values of $\tau_{k}$ are
    \begin{equation*}
        \tau_{1} = \tau_{2} = 0,\quad\tau_{3} = 0.0836,\quad \tau_{4} = 0.2878,\quad
        \tau_{5} = 0.8160,
    \end{equation*}
and  $\alpha_{k}$, $b_{j}^{(k)}$ are listed in Table \ref{BDFk}.
\end{lemma}

Next, we are in the position of providing the stability of the scheme \eqref{semi-discrete}. In particular, we have the following estimate.
\begin{theorem}\label{stability}
Let $\tau_k$ be given as in Lemma \ref{lemma1}. The semi-discrete BDF-$k$ ($1\le k \le 5$) scheme \eqref{semi-discrete} is stable  with $\lambda \geq  0$ in the sense that, 
  \begin{equation}\label{eq:semidis-sta}
            \begin{aligned}    &\lambda_{min}\Vert\varphi^{N+1}\Vert_{L_{\omega_{\alpha,0}}^{2}}^{2} + \frac{\Delta t}{2}(1-\tau_{k}^{2})\sum\limits_{n = k-1}^{N}\left(\Vert\varphi^{n+1}\Vert_{L_{\omega_{\alpha,1}}^{2}}^{2}  + \lambda\vert\phi^{n+1}\vert^{2}\right)\\
            &\leq\frac{1}{1 - \tau_{k}^{2}}\bigg[ \lambda_{max}\alpha 
 +\frac{T}{2}(1-\alpha) \bigg]\vert\phi_{0}\vert^{2},
            \end{aligned}
        \end{equation}
    where $\phi_{0}$ is the initial value,   $\tau_{k}$ are list in Lemma \ref{lemma1}, $\lambda_{min}$ and  $\lambda_{max}$ are the minimum eigenvalue and the maximum eigenvalue of the positive definite symmetric matrix $G = (g_{ij})$, respectively, $\phi^{n} = \mathcal{C}[\varphi^{n}]$ is the solution of the original fractional problem at time $t_{n}$.
    \end{theorem}

    \begin{proof}
         Multiplying both sides of \eqref{semi-discrete} by $(\varphi^{n+1} - \tau_{k}\varphi^{n})\omega_{\alpha}\Delta t$, then integrating with respect to $\theta$ from 0 to 1, we obtain
        \begin{equation}\label{semi-discrete_inner}
        \begin{aligned}
&\left(\mathcal{L}\varphi^{n+1}, \varphi^{n+1} - \tau_{k}\varphi^{n}\right)_{L_{\omega_{\alpha,0}}^{2}} + \Delta t \left(\varphi^{n+1}, \varphi^{n+1}
- \tau_{k}\varphi^{n}\right)_{L_{\omega_{\alpha,1}}^{2}} \\
&+ \Delta t\lambda\left(\mathcal{C}[\varphi^{n+1}],\varphi^{n+1}
- \tau_{k}\varphi^{n}\right)_{L_{\omega_{\alpha}}^{2}}= \Delta t\left(\phi_{0},  \varphi^{n+1}-\tau_{k}\varphi^{n}\right)_{L_{\omega_{\alpha,1}}^{2}}.
            \end{aligned}
        \end{equation}
For the second term of \eqref{semi-discrete_inner}, it follows from $2a(a-b)= a^{2} - b^{2} + (a - b)^{2} $ and $0\leq\tau_{k}<1$ that 
    \begin{equation}\label{stability_term2}
        \begin{aligned}
             &2\left(\varphi^{n+1}, \varphi^{n+1} - \tau_{k}\varphi^{n}\right)_{L_{\omega_{\alpha,1}}^{2}}\\
            & = (\Vert\varphi^{n+1}\Vert_{L_{\omega_{\alpha,1}}^{2}}^{2} -  \tau_{k}^{2}\Vert\varphi^{n}\Vert_{L_{\omega_{\alpha,1}}^{2}}^{2}) + (\Vert\varphi^{n+1} - \tau_{k}\varphi^{n}\Vert^{2}_{L_{\omega_{\alpha,1}}^{2}})\\
            & = (1-\tau_{k}^{2})\Vert\varphi^{n+1}\Vert_{L_{\omega_{\alpha,1}}^{2}}^{2} + \tau_{k}^{2} \left(\Vert\varphi^{n+1}\Vert_{L_{\omega_{\alpha,1}}^{2}}^{2} - \Vert\varphi^{n}\Vert_{L_{\omega_{\alpha,1}}^{2}}^{2}\right)
             + \Vert\varphi^{n+1} - \tau_{k}\varphi^{n}\Vert^{2}_{L_{\omega_{\alpha,1}}^{2}}.
        \end{aligned}
    \end{equation}

    Similarly, for the third term of \eqref{semi-discrete_inner}, we have
\begin{equation}\label{stability_term3}
\begin{aligned}
    &2\left(\mathcal{C}[\varphi^{n+1}], \varphi^{n+1} - \tau_{k}\varphi^{n}\right)_{L_{\omega_{\alpha}}^{2}} = 2\mathcal{C}[\varphi^{n+1}]\left( \mathcal{C}[\varphi^{n+1}] - \tau_{k}\mathcal{C}[\varphi^{n}]\right)\\
    & = (1-\tau_{k}^{2})\vert\phi^{n+1}\vert^{2} + \tau_{k}^{2}\left(\vert\phi^{n+1}\vert^{2} - \vert\phi^{n}\vert^{2}\right) + \vert\phi^{n+1} - \tau_{k}\phi^{n}\vert^{2},
\end{aligned}
\end{equation}
where we use equality \eqref{eq:transform}.
Thus, by using Cauchy-Schwarz and Young inequalities, we deduce from \eqref{eq:lemma1}, \eqref{semi-discrete_inner}-\eqref{stability_term3} that 
       \begin{equation}\label{combining}
       \begin{aligned}
            &\sum\limits_{i,j = 1}^{k}g_{ij}\left(\varphi^{n+1+i-k},\varphi^{n+1+j-k}\right)_{L_{\omega_{\alpha,0}}^{2}} - \sum\limits_{i,j = 1}^{k}g_{ij}\left(\varphi^{n+i-k},\varphi^{n+j-k}\right)_{L_{\omega_{\alpha,0}}^{2}} \\
            &+ \frac{\Delta t}{2}(1-\tau_{k}^{2})\Vert\varphi^{n+1}\Vert_{L_{\omega_{\alpha,1}}^{2}}^{2} + \frac{\Delta t\tau_{k}^{2}}{2} \left(\Vert\varphi^{n+1}\Vert_{L_{\omega_{\alpha,1}}^{2}}^{2} - \Vert\varphi^{n}\Vert_{L_{\omega_{\alpha,1}}^{2}}^{2}\right) \\
            &+\frac{\Delta t\lambda}{2}(1-\tau_{k}^{2})\vert\phi^{n+1}\vert^{2} + \frac{\Delta t\tau_{k}^{2}\lambda}{2} \left(\vert\phi^{n+1}\vert^{2} - \vert\phi^{n}\vert^{2}\right) 
            \leq\frac{\Delta t}{2}\Vert\phi_{0}\Vert_{L_{\omega_{\alpha,1}}^{2}}^{2},
       \end{aligned}
       \end{equation}
where we use the equality \eqref{eq:transform}. Taking the summation of \eqref{combining} for $n$ from $k-1$ to $N$ and using \eqref{eq:alp}, we obtain 
         \begin{equation}\label{semi-discrete 21}
       \begin{aligned}
&\lambda_{min}\Vert\varphi^{N+1}\Vert_{L_{\omega_{\alpha,0}}^{2}}^{2}+ \frac{\Delta t}{2}(1 - \tau_{k}^{2})\sum\limits_{n = k-1}^{N}\Vert\varphi^{n+1}\Vert_{L_{\omega_{\alpha,1}}^{2}}^{2}\\
&+ \frac{\Delta t\tau_{k}^{2}}{2}\Vert\varphi^{N+1}\Vert_{L_{\omega_{\alpha,1}}^{2}}^{2}+ \frac{\Delta t\lambda}{2}(1-\tau_{k}^{2})\sum\limits_{n=k-1}^{N}\vert\phi^{n+1}\vert^{2} + \frac{\Delta t\lambda}{2}\tau_{k}^{2}\vert\phi^{N+1}\vert^{2}\\
            &\leq   \bigg[ \lambda_{max}\alpha 
 +\frac{T}{2}(1-\alpha) \bigg]\vert\phi_{0}\vert^{2} +  \frac{\Delta t\tau_{k}^{2}}{2}\Vert\varphi^{k-1}\Vert_{L_{\omega_{\alpha,1}}^{2}}^{2} + \frac{\Delta t\lambda}{2}\tau_{k}^{2}\vert\phi^{k-1}\vert^{2},
       \end{aligned}
       \end{equation}
where $\lambda_{min}, \lambda_{max}$ are the minimum eigenvalue and maximum eigenvalue of the positive definite symmetric matrix $G = (g_{ij})$, respectively.

For the last two terms of the above estimate, we bound them by utilizing the low-order BDF-$k$ scheme similarly:
         \begin{equation}\label{semi-discrete 4}
             \frac{\Delta t\tau_{k}^{2}}{2}\Vert\varphi^{k-1}\Vert_{L_{\omega_{\alpha,1}}^{2}}^{2} +\frac{\Delta t\lambda\tau_{k}^{2}}{2}\vert\phi^{k-1}\vert^{2}\leq\frac{\tau_{k}^{2}}{1 - \tau_{k}^{2}}\bigg[ \lambda_{max}\alpha 
 +\frac{T}{2}(1-\alpha) \bigg]\vert\phi_{0}\vert^{2}.
         \end{equation}
Then the estimate \eqref{eq:semidis-sta} follows by \eqref{semi-discrete 21} and \eqref{semi-discrete 4}.
\end{proof}
\subsubsection{Spectral collocation method for the $\theta$ direction}
We now show the spectral approximation for the $\theta$ direction and consequently give the fully discretization scheme.
Let $\{\theta_j\}_{j=0}^M$ be the Jacobi-Gauss points in $[0,1]$ with respect to the weight $\Gamma(\alpha) \Gamma(1-\alpha)\omega_{\alpha}(\theta) = \theta^{-\alpha}(1-\theta)^{\alpha - 1}$, and $\{h_j(\theta)\}_{j=0}^M$ be the Lagrange interpolation functions with respect to the Jacobi-Gauss points $\{\theta_j\}_{j=0}^M$, respectively. We then approximate the function $\varphi^n(\theta)$ by 
\begin{equation*}
    \varphi^{n}_{M}(\theta) = \sum\limits_{j=0}^{M}\varphi^{n}(\theta_{j})h_{j}(\theta),
\end{equation*}
and compute $\phi^n_M:=\mathcal{C}[\varphi^{n}_M]$ by using the Gauss-Jacobi quadrature, namely, 
\begin{equation*}\label{Cphi}
\begin{aligned}
     \mathcal{C}[\varphi^{n}_M] = \int_{0}^{1} \varphi^{n}_M(\theta)\omega_{\alpha}(\theta)d\theta
    =  \sum\limits_{j = 0}^{M}\varphi_{M}^{n}(\theta_{j})\omega_{j},
\end{aligned}
\end{equation*}
where $\omega_{j} = \int_{0}^{1}h_{j}(\theta)\omega_{\alpha}(\theta)d\theta,\, j = 0,\ldots,M$ are the corresponding Jacobi-Guass weights.

Now we give the spectral collocation scheme for the semi-discrete problem \eqref{semi-discrete}: for $0\le s\le M$, 
\begin{equation}\label{fully scheme}
     \frac{(1 - \theta_{s})}{\Delta t}\mathcal{L}\varphi_{M}^{n+1}(\theta_{s}) + \theta_{s} \varphi^{n+1}_{M}(\theta_{s}) = -\lambda\mathcal{C}[\varphi_{M}^{n+1}] + \theta_{s}\phi_{0}.
\end{equation}
Denote $P_M$ as the set of all algebraic polynomials of degree $\le M$, we have 
\begin{equation*}
    P_{M} = \text{span}\{ h_{j}(\theta):0\leq j \leq M\}.
\end{equation*}
Note that $\forall~ q_M\in P_M$, we have $(1-\theta) \varphi_M q_M,\; \theta \varphi_M q_M \in P_{2M+1}$, then the spectral collocation formula \eqref{fully scheme} is equivalent to the following Galerkin form: Find $\varphi_{M}^{n+1} \in P_M$, such that 
\begin{equation}\label{fully discrete formula}
\begin{aligned}
    &\frac{1}{\Delta t}\left(\mathcal{L}\varphi^{n+1}_{M},q_{M}\right)_{L_{\omega_{\alpha,0}}^{2}} + (\varphi_{M}^{n+1},q_{M})_{L_{\omega_{\alpha,1}}^{2}} + \lambda(\mathcal{C}[\varphi_{M}^{n+1}], q_{M})_{L_{\omega_{\alpha}}^{2}} \\
    &=  (\phi_{0},q_{M})_{L_{\omega_{\alpha,1}}^{2}}\quad \forall q_M \in P_M.
\end{aligned}
\end{equation}

Therefore, by using the same argument used for Theorem \ref{stability}, we conclude that the fully discretization scheme \eqref{fully scheme} is stable. 
\begin{theorem}\label{stability of fully discrete scheme}
    For $1\leq k\leq 5$, the fully-discrete scheme \eqref{fully scheme} with $\lambda \geq 0$ is stable in the sense that,
    \begin{equation}
         \begin{aligned}    &\lambda_{min}\Vert\varphi_{M}^{N+1}\Vert_{L_{\omega_{\alpha,0}}^{2}}^{2} + \frac{\Delta t}{2}(1-\tau_{k}^{2})\sum\limits_{n = k-1}^{N}\left(\Vert\varphi^{n+1}_{M}\Vert_{L_{\omega_{\alpha,1}}^{2}}^{2}  + \lambda\left\vert\mathcal{C}[\varphi_{M}^{n+1}]\right\vert^{2}\right)\\
            &\leq\frac{1}{1 - \tau_{k}^{2}}\bigg[ \lambda_{max}\alpha 
 +\frac{T}{2}(1-\alpha) \bigg]\vert\phi_{0}\vert^{2},
            \end{aligned}
    \end{equation}
     where $\phi_{0}$ is the initial value, $\tau_{k}$ are list in Lemma \ref{lemma1},  $\lambda_{min}$ and  $\lambda_{max}$ are the minimum eigenvalue and the maximum eigenvalue of the positive definite symmetric matrix $G = (g_{ij})$, respectively.
\end{theorem}
\subsubsection{Implementation}
Now let us briefly give a description for the efficient implementation. 
We write the scheme \eqref{fully scheme} into the following matrix form 
\begin{equation}\label{eq:system}
    \mathbf{A} \mathbf{\Phi}^{n+1} = \sum\limits_{j=0}^{k-1}\mathbf{B}_{j} \mathbf{\Phi}^{n-j} + \mathbf{F},
\end{equation}
where 
\begin{equation*}
\begin{aligned}
     &\mathbf{A} = {\rm{diag}}(A) + \Delta t\lambda \mathbf{J}_{M+1}\mathbf{W},\quad A = \left[\alpha_{k}(1-\theta_{0}) + \Delta t\theta_{0},\ldots,\alpha_{k}(1-\theta_{M}) + \Delta t\theta_{M}\right]^{T},\\
     &\mathbf{B}_{j} = b_{j}^{(k)} {\rm{diag}}(B),\quad
    B = \left[1-\theta_{0},\ldots,1-\theta_{M}\right]^{T},\quad \mathbf{F} = \Delta t\phi_{0}\left[\theta_{0},\ldots,\theta_{M}\right]^{T},\\
     &\mathbf{W} = {\rm{diag}}(W),\quad W = \left[\omega_{0},\ldots,\omega_{M}
     \right]^{T},\quad  \mathbf{\Phi}^j = \left[\varphi^j(\theta_{0}),\ldots,\varphi^j(\theta_{M})\right]^{T},
\end{aligned}
\end{equation*}
and $\Delta t$ is time step size, $\mathbf{J}_{M+1}$ is the all-one $(M+1)\times (M+1)$ matrix.

Due to the stability of the scheme \eqref{fully scheme}, we have that $\mathbf{A}$ is non-singular. Then we have the following characteristic decomposition of $\mathbf{A}$:
\begin{equation*}\label{CD}
    \mathbf{X\Lambda} \mathbf{X}^{-1} = \mathbf{A},
\end{equation*}
where $\mathbf{X}$ is the matrix whose columns are the corresponding eigenvectors of $\mathbf{A}$ and $\mathbf{\Lambda}$ is the diagonal
matrix whose diagonal entries are the eigenvalues of $\mathbf{A}$, respectively.
Multiply both sides of \eqref{eq:system} with $\mathbf{X}^{-1}$ and define the vectors
\begin{equation}\label{eq:YX}
    \mathbf{Y} = \mathbf{X}^{-1} \mathbf{\Phi}^{n+1},
    \quad \mathbf{\hat{Y}} = \mathbf{\Lambda Y},
\end{equation}
we have 
\begin{equation*}\label{fast scheme}
    \begin{aligned}
        \mathbf{\hat{Y}} = \mathbf{\Lambda Y} = \mathbf{X}^{-1}\left(\sum\limits_{j=0}^{k-1}\mathbf{B}_{j} \mathbf{\Phi}^{n-j} + \mathbf{F}\right).
    \end{aligned}
\end{equation*}

Thus, we shall obtain $\mathbf{\hat{Y}}$ by using the above equation. Consequently, we compute $\mathbf{Y}$ and $\mathbf{\Phi}^{n+1} = \mathbf{XY}$ by using \eqref{eq:YX}.
We summarize the implementation in the following Algorithm \ref{fast algorithm}:
\begin{algorithm}[htbp] 
	\caption{Efficient implementation for the fully-discrete scheme \eqref{fully scheme}}     
	 \label{fast algorithm}  
	\begin{algorithmic}[1]
	\Require $\mathbf{A},\, \mathbf{B}_j,\,\mathbf{\Phi}^{n-j},\, j=0,\ldots, k-1$ 
     \Ensure  $\mathbf{\Phi}^{N+1}$ 
     \State  Calculate $\mathbf{X}^{-1}$ and $\mathbf{\Lambda}$ by $\mathbf{X\Lambda} \mathbf{X}^{-1} = \mathbf{A}$
    \For {$n=k-1,...,N$ }
         \State We compute $\mathbf{\hat{Y}}$ by $\mathbf{\hat{Y}}  = \mathbf{X}^{-1}\left(\sum\limits_{j=0}^{k-1}\mathbf{B}_{j} \mathbf{\Phi}^{n-j} + \mathbf{F}\right)$
        \State We compute $\mathbf{Y}$ by $\mathbf{\hat{Y}} = \mathbf{\Lambda Y}$
        \State We compute $\mathbf{\Phi}^{n+1}$ by $\mathbf{\Phi}^{n+1} = \mathbf{X}\mathbf{Y} $
  \EndFor
  \State\Return  $\mathbf{\Phi}^{N+1}$ 
   \end{algorithmic} 
\end{algorithm}

\begin{remark}
We observe from the above that \emph{the computational cost and the storage requirement of Algorithm \ref{fast algorithm} for the present scheme \eqref{fully scheme} are $O(N)$ and $O(1)$}, respectively, for a \emph{fixed} value of $M$, i.e., a \emph{fixed} degree of freedom for the spectral approximation in the $\theta$ direction.
\end{remark}

\begin{remark}
We point out here that a \emph{small fixed} value of $M$ (says $M=30$) is enough to deliver a high accuracy since we observe from numerical results (see Figure \ref{Fig:ex1:theta} and \ref{4a}) that the spectral accuracy is achieved for the spectral discretization in the $\theta$ direction for both linear and nonlinear problems, moreover, this is guaranteed by the theoretical result for the linear case, i.e., the solutions of the EPDE have high regularity with respect to $\theta$. 
\end{remark}
\subsection{Error estimate}\label{sec:errorest}
We now give the error analysis for the proposed scheme \eqref{fully scheme}.

\subsubsection{Preliminary}
We first provide several preliminary knowledge for the spectral approximation.
Let $\omega^{a,b}(\theta) = (1-\theta)^{a}\theta^{b}, a,b > -1$ and $\Omega = (0,1)$. We define the Jacobi-weighted Sobolev space:
\begin{equation*}
    B_{\alpha-1,-\alpha}^{m}(\Omega) = \left\{v:\frac{d ^{l}v}{d \theta^{l}} (\theta)\in L_{\omega^{\alpha-1+l,-\alpha+l}}^{2}(\Omega),\quad 0\leq l\leq m \right\}.
\end{equation*}
Then for any $v\in B_{\alpha-1,-\alpha}^{m}(\Omega),\, 0\leq m\leq M$, up to a constant, it holds~(see \cite[Theorem 3.35]{SM})
\begin{equation}\label{projection}
    \Vert\Pi_{M}^{\alpha-1,-\alpha}v - v\Vert_{L_{\omega^{\alpha-1,-\alpha}}^{2}}\leq C M^{-m}\Vert\frac{d^{m}v}{d \theta^{m}}\Vert_{L_{\omega^{\alpha-1+m,-\alpha+m}}^{2}},
\end{equation}
where $C$ is a constant and $\Pi_{M}^{\alpha-1,-\alpha}: L^2_{\omega^{\alpha-1, -\alpha}}(\Omega) \rightarrow P_M(\Omega)$ is the $L^2_{\omega^{\alpha-1, -\alpha}}$ projection. 

Now we define the bilinear form 
\begin{equation}\label{eq:bilinearform}
    \mathcal{A}(\varphi,q) = \left(\varphi,q\right)_{L_{\omega_{\alpha}}^{2}} + \lambda\left(\int_{0}^{1}\varphi\omega_{\alpha}d\theta,q\right)_{L_{\omega_{\alpha}}^{2}}
    = \left(\varphi,q\right)_{L_{\omega_{\alpha}}^{2}} + \lambda \left(\mathcal{C}[\varphi],q\right)_{L_{\omega_{\alpha}}^{2}}
\end{equation}
and $X := \{\varphi|\Vert\varphi\Vert_{X}\leq\infty\}$ with the associated norm given by 
\begin{equation*}
    \Vert\varphi\Vert_{X}^{2} := \mathcal{A}(\varphi,\varphi) =  \Vert\varphi\Vert_{L_{\omega_{\alpha}}^{2}}^{2} + \lambda\left(\int_{0}^{1}\varphi\omega_{\alpha}d\theta\right)^{2}.
\end{equation*}

Consider the following weak problem: Find $\varphi\in X$, such that
\begin{equation*}
    \mathcal{A}(\varphi,q) = (f,q) \quad \forall q\in X.
\end{equation*}
Then, for $f\in X^{*}$, where $X^{*}$ is the dual space of $X$,  we have that the above problem admits a unique solution since 
\begin{equation*}
   \Vert\varphi\Vert_{L_{\omega_{\alpha}}^{2}}^{2} \le \mathcal{A}(\varphi, \varphi) \le \bar{C} \Vert\varphi\Vert_{L_{\omega_{\alpha}}^{2}}^{2}, \quad \mathcal{A}(\varphi, q) \le \hat{C}\Vert\varphi\Vert_{L_{\omega_{\alpha}}^{2}} \Vert q \Vert_{L_{\omega_{\alpha}}^{2}},
\end{equation*}
where $\bar{C},\, \hat{C}$ are two constants. The first equation in the above also indicates that $\Vert\varphi\Vert_{L_{\omega_{\alpha}}^{2}}\approx \Vert\varphi\Vert_{X}.$

Now let us define the projection $\Pi_{M}:X\rightarrow P_{M}$ such that
\begin{equation}\label{projection bilinear}
    \mathcal{A}(\varphi - \Pi_{M}\varphi,q_{M}) = 0\quad \forall q_{M}\in P_{M}.
\end{equation}
Immediately, one can easily obtain the following estimates.
\begin{lemma}\label{projection lemma}
For $\varphi\in B_{\alpha - 1,-\alpha}^{m}(\Omega)$, we have 
    \begin{equation}\label{eq:prj:01}
         \Vert \varphi - \Pi_{M}\varphi \Vert_{X}\leq C M^{-m}\Vert\frac{\partial^{m}\varphi}{\partial \theta^{m}}\Vert_{L_{\omega^{\alpha-1+m,-\alpha+m}}^{2}}.
\end{equation}
\end{lemma}
\begin{proof} We obtain from \eqref{projection bilinear} that 
\begin{equation*}
\begin{aligned}
     \Vert\varphi - \Pi_{M}\varphi\Vert_{X}^{2} &= \mathcal{A}(\varphi - \Pi_{M}\varphi,\varphi - \Pi_{M}\varphi)\\
     & = \mathcal{A}(\varphi - \Pi_{M}\varphi,\varphi - \nu_{M})\\
     &\le C\Vert\varphi - \Pi_{M}\varphi\Vert_{X}\Vert\varphi - \nu_{M}\Vert_{X} \quad \forall \nu_{M}\in P_{M}.
\end{aligned}    
\end{equation*}
Then we take $\nu_{M} = \Pi_{M}^{\alpha-1,-\alpha}\varphi\in P_{M}$ in the above equation to yield
\begin{equation*}
\begin{aligned}
    \Vert\varphi - \Pi_{M}\varphi\Vert_{X}\le C\Vert\varphi - \Pi_{M}^{\alpha-1,-\alpha}\varphi\Vert_{X}\le \tilde{C}\Vert\varphi - \Pi_{M}^{\alpha-1,-\alpha}\varphi\Vert_{L_{\omega_{\alpha}}^{2}},
\end{aligned}
\end{equation*}
where $\tilde{C}$ is a constant.
Note that $\Gamma(\alpha)\Gamma(1-\alpha)\omega_{\alpha}(\theta) = \omega^{\alpha-1,-\alpha}(\theta)$. Hence, the estimate \eqref{eq:prj:01} follows by  \eqref{projection} and the above equation.  
\end{proof}

\subsubsection{Error estimate}
Now we proceed to derive the error estimates for the fully discretization scheme \eqref{fully scheme}.
To this end, we develop a Galerkin form for EPDE \eqref{equation} with $F(t,\phi) = -\lambda \phi$ by adding $c_0(\theta) \varphi$ on both sides: Find $\varphi \in L^2_{\omega_{\alpha}}(\Omega)$, such that 
\begin{equation}\label{CP}
    \begin{aligned}
        \left(\frac{\partial \varphi}{\partial t} (t),q\right)_{L_{\omega_{\alpha,0}}^{2}} + \mathcal{A}(\varphi(t),q) =  \left(\varphi(t),q\right)_{L_{\omega_{\alpha,0}}^{2}} + \left(\phi_{0},q\right)_{L_{\omega_{\alpha,1}}^{2}} \quad \forall q \in  L^2_{\omega_{\alpha}}(\Omega),
    \end{aligned}
\end{equation}
where the bilinear form $\mathcal{A}(\cdot,\cdot)$ is defined in \eqref{eq:bilinearform}.


Now we present the main result of this section, that is, the error estimate of the scheme \eqref{fully scheme} as follows:
\begin{theorem}\label{Theorem:fully error estimate}
    Let $\varphi_{M}^{n+1}$ and $\varphi(t_{n+1},\theta)$ be the solutions of \eqref{fully scheme} and \eqref{CP}, respectively.  Assume $\varphi, \frac{\partial \varphi}{\partial t}\in L^{2}(I;B_{\alpha-1,-\alpha}^{m}(\Omega)),\; (1-\theta) \frac{\partial^{k+1} \varphi}{\partial t^{k+1}}\in L^{2}(I;L_{\omega_{\alpha}}^{2}(\Omega))$,
    $0\leq m\leq M$, $1\leq k\leq 5$, if $\Delta t< 1$, then we have 
    \begin{equation}\label{error estimate}
    \begin{aligned}
        &\Vert\varphi^{N+1}_{M} - \varphi(t_{N+1})\Vert_{L_{\omega_{\alpha,0}}^{2}}^{2} + \kappa\Delta t\sum\limits_{n = k-1}^{N}\Vert\varphi^{n+1}_{M} - \varphi(t_{n+1})\Vert_{X}^{2} \\
        \le & C_{3}\bigg[M^{-2m}\int_{0}^{T}\left(\Vert\frac{\partial^{m+1} \varphi}{\partial s\partial\theta^{m}} (s)\Vert_{L_{\omega^{\alpha -1+m, -\alpha+m}}^{2}}^{2} + \Vert\frac{\partial^{m} \varphi}{\partial\theta^{m}} (s)\Vert_{L_{\omega^{\alpha -1+m, -\alpha+m}}^{2}}^{2}\right)ds \\
        & + \Delta t^{2k}\int_{0}^{T}\Vert (1-\theta)\frac{\partial^{k+1}\varphi}{\partial s^{k+1}}(s) \Vert_{L_{\omega_{\alpha}}^{2}}^{2}ds\bigg].
    \end{aligned}
\end{equation}
where $C_{3} = exp(\varrho T)max\left(\frac{d_{max}C}{4\xi\lambda_{min}},\frac{1}{4\xi\lambda_{min}},\frac{C_{1}}{4\xi\lambda{min}}\right)$ with $\varrho = \left(1 - \Delta t\right)^{-1}$,  $\kappa = \frac{1-\tau_{k}^{2}}{2\lambda_{min}}$, $\tau_{k}$ are defined in Lemma \ref{lemma1} and $\lambda_{min}$ is the minimum eigenvalue of the positive definite symmetric matrix $G = (g_{ij})$, $\xi$ is a positive constant.
\end{theorem}
\begin{proof}
Define $\eta_{M}^{n+1} = \varphi_{M}^{n+1} - \Pi_{M}\varphi(t_{n+1}).$
By adding both sides of \eqref{fully discrete formula} (which is equivalent to \eqref{fully scheme}) with $\left(\varphi_{M}^{n+1}, q_{M}\right)_{L_{\omega_{\alpha,0}}^{2}}$ and subtracting the result equation from \eqref{CP} and using \eqref{projection bilinear}, we obtain
\begin{equation*}\label{error:01}
\begin{aligned}
&\frac{1}{\Delta t}\left(\mathcal{L}\eta^{n+1}_{M},q_{M}\right)_{L_{\omega_{\alpha,0}}^{2}} + \mathcal{A}\left(\eta_{M}^{n+1},q_{M}\right) \\
    &= -\frac{1}{\Delta t}\left(\mathcal{L}(\Pi_{M}\varphi(t_{n+1})),q_{M}\right)_{L_{\omega_{\alpha,0}}^{2}}  -  \mathcal{A}\left(\varphi(t_{n+1}),q_{M}\right) + \left(\phi_{0},q_{M}\right)_{L_{\omega_{\alpha,1}}^{2}}\\
    &+ \left(\varphi(t_{n+1}),q_{M}\right)_{L_{\omega_{\alpha,0}}^{2}} + \left(\varphi_{M}^{n+1} - \varphi(t_{n+1}),q_{M}\right)_{L_{\omega_{\alpha,0}}^{2}}\quad \forall q_{M}\in P_{M},
\end{aligned}
\end{equation*}
We then derive from \eqref{CP} and the above equation that
\begin{equation}\label{error:02}
\begin{aligned}
  &\left(\mathcal{L} \eta^{n+1}_{M}, q_{M} \right)_{L_{\omega_{\alpha,0}}^{2}} + \Delta t \mathcal{A}\left(\eta_{M}^{n+1},q_{M}\right) \\
    & = \left(\mathcal{L}(I - \Pi_{M})\varphi(t_{n+1}),q_{M}\right)_{L_{\omega_{\alpha,0}}^{2}} + \left(\Delta t\frac{\partial \varphi}{\partial t}(t_{n+1}) - \mathcal{L}\varphi(t_{n+1}),q_{M}\right)_{L_{\omega_{\alpha,0}}^{2}}\\
    & + \Delta t\left(\eta_{M}^{n+1},q_{M}\right)_{L_{\omega_{\alpha,0}}^{2}} + \Delta t\left((\Pi_{M} - I)\varphi(t_{n+1}),q_{M}\right)_{L_{\omega_{\alpha,0}}^{2}}.
\end{aligned}  
\end{equation}

For the first term of the right hand side of the above equation, note that 
$\mathcal{L}(I - \Pi_{M})\varphi(t_{n+1}) = \sum\limits_{j = 0}^{k-1}d_{j}\int_{t_{n-j}}^{t_{n+1-j}}(I - \Pi_{M})\frac{\partial \varphi}{\partial s} (s,\theta)ds,$
where $d_{j}$ are some fixed and bounded constants (see \cite[Theorem 11.3.4]{Numericalsolution}). Then by Cauchy-Schwarz and Young inequalities, we have 
\begin{equation}\label{error:04}
    \begin{aligned}
        &\left(\mathcal{L}(I - \Pi_{M})\varphi(t_{n+1}),q_{M}\right)_{L_{\omega_{\alpha,0}}^{2}} = \sum\limits_{j = 0}^{k-1}d_{j}\left(\int_{t_{n-j}}^{t_{n+1-j}}(I - \Pi_{M})\frac{\partial \varphi}{\partial s}(s,\theta) ds,q_{M}\right)_{L_{\omega_{\alpha,0}}^{2}}\\
         &\le\frac{1}{4\xi}\sum\limits_{j=0}^{k-1}\vert d_{j}\vert\int_{t_{n-j}}^{t_{n+1-j}}\Vert(I - \Pi_{M})\frac{\partial \varphi}{\partial s} (s)\Vert^{2}_{L_{\omega_{\alpha,0}}^{2}}ds + \xi\sum\limits_{j=0}^{k-1}\vert d_{j}\vert\int_{t_{n-j}}^{t_{n+1-j}}\Vert q_{M}\Vert^{2}_{L_{\omega_{\alpha,0}}^{2}}ds\\
         & \le \frac{d_{max}}{4\xi}\sum\limits_{j=0}^{k-1}\int_{t_{n-j}}^{t_{n+1-j}}\Vert(I - \Pi_{M})\frac{\partial \varphi}{\partial s} (s)\Vert^{2}_{L_{\omega_{\alpha,0}}^{2}}ds + \Delta t\xi\sum\limits_{j=0}^{k-1}\vert d_{j}\vert\Vert q_{M}\Vert^{2}_{L_{\omega_{\alpha,0}}^{2}},
    \end{aligned}
\end{equation}
where $d_{max} = \max{\vert d_{j}\vert, j=0,\ldots, k-1}$, $\xi$ is a positive constant.

For the second term of the right hand side of \eqref{error:02}, again,
it follows from \cite[Theorem 11.3.4]{Numericalsolution} that 
\begin{equation*}\label{error:05}
\begin{aligned}
    &\Delta t\frac{\partial \varphi}{\partial t} (t_{n+1}) - \mathcal{L}\varphi(t_{n+1}) = \sum\limits_{j = 0}^{k-1}c_{j}\int_{t_{n-j}}^{t_{n+1}}(t_{n-j} - s)^{k}\frac{\partial^{k+1}\varphi(s)}{\partial t^{k+1}}ds,
\end{aligned}
\end{equation*}
where $c_{j}$  are also some fixed and bounded constants. Then the second term of the right hand side of \eqref{error:02} is bounded by 
\begin{equation}\label{error:06}
    \begin{aligned}
        & \frac{1}{4\xi\Delta t}\Vert (1-\theta)(\Delta t\frac{\partial \varphi(t_{n+1})}{\partial t} - \mathcal{L}\varphi(t_{n+1}) )\Vert_{L_{\omega_{\alpha}}^{2}}^{2} + \xi\Delta t\Vert q_{M}\Vert_{L_{\omega_{\alpha}}^{2}}^{2}\\
        \le & \frac{C_{1}}{4\xi}\Delta t^{2k}\sum\limits_{j = 0}^{k-1}\int_{t_{n-j}}^{t_{n+1-j}}\Vert (1-\theta)\frac{\partial^{k+1}\varphi(s)}{\partial s^{k+1}}\Vert_{L_{\omega_{\alpha}}^{2}}^{2}ds + \xi\Delta t\Vert q_{M}\Vert_{L_{\omega_{\alpha}}^{2}}^{2},
    \end{aligned}
\end{equation}
where $C_{1}$ is the maximum value of  the linear combination of $\vert c_{j}\vert$. 

We use a similar approach for the last two terms of \eqref{error:02} to obtain 
\begin{equation}\label{error:10}
    \begin{aligned}
        \left(\eta_{M}^{n+1},q_{M}\right)_{L_{\omega_{\alpha,0}}^{2}}\le \lambda_{min}\Vert\eta_{M}^{n+1}\Vert_{L_{\omega_{\alpha,0}}^{2}}^{2} + \frac{1}{4\lambda_{min}}\Vert q_{M}\Vert_{L_{\omega_{\alpha,0}}^{2}}^{2}.
    \end{aligned}
\end{equation}
where $\lambda_{min}$ is the minimum eigenvalue of the positive definite symmetric matrix $G = (g_{ij})$.
\begin{equation}\label{error:11}
    \begin{aligned}
       \left((\Pi_{M} - I)\varphi(t_{n+1}),q_{M}\right)_{L_{\omega_{\alpha,0}}^{2}} 
       \le \frac{1}{4\xi}\Vert(\Pi_{M} - I)\varphi(t_{n+1})\Vert^{2}_{L_{\omega_{\alpha,0}}^{2}} + \xi \Vert q_{M}\Vert^{2}_{L_{\omega_{\alpha,0}}^{2}}.
    \end{aligned}
\end{equation}

By taking $q_{M} = \eta_{M}^{n+1} - \tau_{k}\eta_{M}^{n}$ and using the same argument as in Theorem \ref{stability}, we have from Lemma \ref{lemma1} and \eqref{error:02} - \eqref{error:11} that
\begin{equation*}\label{error:07}
    \begin{aligned}
        &\sum\limits_{ij = 1}^{k}g_{ij}\left[\left(\eta_{M}^{n+1+i-k},\eta_{M}^{n+1+j-k}\right)_{L_{\omega_{\alpha,0}}^{2}} - \left(\eta_{M}^{n+i-k},\eta_{M}^{n+j-k}\right)_{L_{\omega_{\alpha,0}}^{2}}\right]\\
        & + \frac{(1-\tau_{k}^{2})\Delta t}{2}\Vert\eta_{M}^{n+1}\Vert_{X}^{2} + \frac{\Delta t\tau_{k}^{2}}{2}\left(\Vert\eta_{M}^{n+1}\Vert_{X}^{2} - \Vert\eta_{M}^{n}\Vert_{X}^{2}\right) + \frac{\Delta t}{2}\Vert\eta_{M}^{n+1} - \tau_{k}\eta_{M}^{n}\Vert_{X}^{2}\\
        &\le \Delta t\lambda_{min}\Vert\eta_{M}^{n+1}\Vert^{2}_{L_{\omega_{\alpha,0}}^{2}} + \frac{d_{max}}{4\xi}\sum\limits_{j=0}^{k-1}\int_{t_{n-j}}^{t_{n+1-j}}\Vert(I - \Pi_{M})\frac{\partial \varphi(s)}{\partial s}\Vert^{2}_{L_{\omega_{\alpha,0}}^{2}}ds \\
        &  + \frac{C_{1}\Delta t^{2k}}{4\xi}\sum\limits_{j = 0}^{k-1}\int_{t_{n-j}}^{t_{n+1-j}}\Vert (1-\theta)\frac{\partial^{k+1}\varphi(s)}{\partial s^{k+1}}\Vert_{L_{\omega_{\alpha}}^{2}}^{2}ds + \frac{\Delta t}{4\xi}\Vert(\Pi_{M} -I)\varphi(t_{n+1})\Vert^{2}_{L_{\omega_{\alpha,0}}^{2}}\\
        & + \Delta t\xi\left(\sum\limits_{j = 0}^{k-1}\vert d_{j}\vert + 2\right)\Vert\eta_{M}^{n+1} - \tau_{k}\eta_{M}^{n}\Vert_{L_{\omega_{\alpha,0}}^{2}}^{2} + \frac{\Delta t}{4\lambda_{min}}\Vert\eta_{M}^{n+1} - \tau_{k}\eta_{M}^{n}\Vert_{L_{\omega_{\alpha,0}}^{2}}^{2}.
    \end{aligned}
\end{equation*}
Note that $\Vert v\Vert_{L_{\omega_{\alpha,0}}^{2}} \le \Vert v\Vert_{L_{\omega_{\alpha}}^{2}}\approx \Vert v \Vert_{X},\, \forall \, v \in X.$ Taking the summation of the above equation for $n$ from $k-1$ to $N$ and setting the positive constant $\xi \le (\frac{1}{2} - \frac{1}{4\lambda_{min}})/{\left(\sum\limits_{j = 0}^{k-1}\vert d_{j}\vert + 2\right)}$ yields
\begin{equation*}
    \begin{aligned}
        &\lambda_{min} \Vert \eta_{M}^{N+1} \Vert_{L_{\omega_{\alpha,0}}^{2}}^{2} + \frac{(1-\tau_{k}^{2})\Delta t}{2}\sum\limits_{n = k-1}^{N}\Vert\eta_{M}^{n+1}\Vert_{X}^{2} + \frac{\tau_{k}^{2}\Delta t}{2}\Vert\eta_{M}^{N+1}\Vert_{X}^{2}\\
        &\le\Delta t\lambda_{min}\sum\limits_{n = k-1}^{N}\Vert\eta_{M}^{n+1}\Vert^{2}_{L_{\omega_{\alpha,0}}^{2}} + \frac{d_{max}C}{4\xi}\int^{T}_{0}\Vert(I - \Pi_{M})\frac{\partial \varphi}{\partial s} (s)\Vert^{2}_{X}ds + \frac{\tau_{k}^{2}\Delta t}{2}\Vert\eta_{M}^{k-1}\Vert_{X}^{2}\\
        & + \frac{C_{1}}{4\xi}\Delta t^{2k}\int_{0}^{T}\Vert (1-\theta)\frac{\partial^{k+1}\varphi}{\partial s^{k+1}} (s)\Vert_{L_{\omega_{\alpha}}^{2}}^{2}ds +  \frac{1}{4\xi}\int_{0}^{T}\Vert(\Pi_{M} -I)\varphi(s)\Vert_{X}^{2}ds.
    \end{aligned}
\end{equation*}
We then obtain the following estimate by applying the discrete Gronwall Lemma (see \cite[Lemma B.10]{SM}) with $\Delta t<1$,
\begin{equation*}\label{error:08}
    \begin{aligned}
        &\Vert\eta_{M}^{N+1}\Vert_{L_{\omega_{\alpha,0}}^{2}}^{2} + \frac{(1-\tau_{k}^{2})\Delta t}{2\lambda_{min}}\sum\limits_{n = k-1}^{N}\Vert\eta_{M}^{n+1}\Vert_{X}^{2} + \frac{\tau_{k}^{2}\Delta t}{2\lambda_{min}}\Vert\eta_{M}^{N+1}\Vert_{X}^{2}\\
        &\le exp(\varrho T)C_{2}\bigg[\int^{T}_{0}\Vert(I - \Pi_{M})\frac{\partial \varphi}{\partial s} (s)\Vert^{2}_{X}ds  +   \int_{0}^{T}\Vert(\Pi_{M} -I)\varphi (s)\Vert^{2}_{X}ds\\
        &+ \Delta t^{2k}\int_{0}^{T}\Vert (1-\theta)\frac{\partial^{k+1}\varphi}{\partial s^{k+1}} (s)\Vert_{L_{\omega_{\alpha}}^{2}}^{2}ds\bigg],
    \end{aligned}
\end{equation*}
where $\varrho = (1 - \Delta t)^{-1}$, $C_{2} = max\left(\frac{d_{max}C}{4\xi\lambda_{min}},\frac{1}{4\xi\lambda_{min}},\frac{C_{1}}{4\xi\lambda{min}}\right)$. 
Thus, the error estimate \eqref{error estimate} follows by using Lemma \ref{projection lemma} and the triangle inequality
\begin{equation*}
    \begin{aligned}
        \Vert\varphi_{M}^{n+1} - \varphi(t_{n+1})\Vert_{X}\le \Vert\eta_{M}^{n+1}\Vert_{X} + \Vert(\Pi_{M} - I)\varphi(t_{n+1})\Vert_{X}.
    \end{aligned}
\end{equation*}
\end{proof}
\begin{remark}
For the nonlinear case, if $F(t,\phi)$ satisfies the Lipschitz condition with respect to the second variable, then we can similarly derive the error estimate by applying a spectral Galerkin scheme for the $\theta$ direction.
We shall show this in a future work.
\end{remark}

We now show a more delicate estimate \emph{without the assumptions of the regularities} on $\varphi,\; \frac{\partial \varphi}{\partial t}$. Instead, we analytically establish the regularity result stated in the following Theorem.

\begin{theorem}\label{thm:regularity}
Let $\varphi$ be the solution of  $\rm{EPDE}$ \eqref{equation} with $F(t,\mathcal{C}[\varphi](t)) = -\lambda\mathcal{C}[\varphi](t) + f(t)$, $\lambda \ge 0$. If $f(t), \, \frac{df}{dt}(t) \in L^\infty([0,T]),\; \frac{\partial \varphi^{k+1}}{\partial t\partial \theta^{k}}(0,\theta)\in L_{\omega^{2-\alpha+k,1-\alpha+k}}^{2}(\Omega),\, 0\le k\le m,\; \forall \,m < +\infty$,  then we have 
\begin{equation}\label{regularity}
        \frac{\partial^{m} \varphi}{\partial \theta^{m}} \in L^{2}(0,T;L_{\omega^{1-\alpha+m,-\alpha+m}}^{2}(\Omega)), \; \frac{\partial^{m+1} \varphi}{\partial t\partial \theta^{m}}\in L^{2}(0,T;L_{\omega^{1-\alpha+m,-\alpha+m}}^{2}(\Omega)).
\end{equation}
\end{theorem}

The proof is given in the Appendix.

The above result indicates that the EPDE exhibits high regularity with respect to $\theta$. Consequently, by combining Theorem \ref{Theorem:fully error estimate} and Theorem \ref{thm:regularity}, we have the following result. 
\begin{theorem}\label{Theorem:fully error estimate2}
    Let $\varphi_{M}^{n+1}$ and $\varphi(t_{n+1},\theta)$ be the solutions of the problem \eqref{fully scheme} and \eqref{CP}, respectively. Assume that the conditions given in Theorem \ref{thm:regularity} hold and $(1-\theta) \frac{\partial^{k+1} \varphi}{\partial t^{k+1}}\in L^{2}(I;L_{\omega_{\alpha}}^{2}(\Omega))$. Then for $1\leq k\leq 5$ and $m<\infty$, it holds for $\Delta t< 1,$
    \begin{equation}\label{error estimate2}
    \begin{aligned}
        &\Vert\varphi^{N+1}_{M} - \varphi(t_{N+1})\Vert_{L_{\omega_{\alpha,0}}^{2}}^{2} + \kappa\Delta t\sum\limits_{n = k-1}^{N}\Vert\varphi^{n+1}_{M} - \varphi(t_{n+1})\Vert_{X}^{2} \\
        \le & C(C_{3}, f, \phi_0, k)\left(M^{-2m} + \Delta t^{2k}\right),
    \end{aligned}
\end{equation}
where $\kappa,\; C_3$ are given in Theorem \ref{Theorem:fully error estimate}.
\end{theorem}

\begin{remark}
Theorem \ref{Theorem:fully error estimate2} indicates that we have exponential convergence for the spectral approximation in the $\theta$ direction. This means that we only need a small number of nodes for the parameter $\theta$ space to obtain a high accuracy.
\end{remark}
\section{Numerical test}\label{sec:4}
We now present several numerical examples to demonstrate the effectiveness of the present algorithm and verify the theoretical result concerning the stability and convergence.

\subsection{Stable  region}\label{sec:6}
We begin by showing the stable region of the present numerical scheme for the linear problem
\begin{equation*}
    \frac{d^{\alpha}\phi}{dt^{\alpha}} (t) = -\lambda \phi(t),\quad t\in [0,T].
\end{equation*}
We set the numerical scheme \eqref{fully scheme} as $\mathbb{E}$. We determine the ``stable region" to quantify when the scheme \eqref{fully scheme} is stable for the above problem, and we define the ``stable region" of \eqref{fully scheme} as 
$$\widehat{\Omega} := \{\sigma:\rho(\mathbb{E}(\sigma))<1,\sigma\in \mathbb{C}\}, \text{ where } \sigma = -\Delta t \lambda = x + {\rm i} y.  $$

For the sake of simplicity, here we only show the contours of the BDF-3 scheme. We set $T = 1, \,N = 10^{2},\, M = 30$. We show in Figure \ref{Implicit contours} the contours of $\rho(\sigma)$ for different values of fractional order $\alpha = 0.2,\,0.4,\,0.6,\,0.8$. Observe that the stable region for the BDF-$3$ scheme \eqref{fully scheme} decreases as the values of fractional order $\alpha$ increase.
In addition, it shows that the scheme is stable while $\lambda >0$, which is consistent with our theoretical result (i.e., Theorem \ref{stability of fully discrete scheme}). 
\begin{figure}[htbp]
  \centering
  \subfigure[$\alpha = 0.2$]
  {\begin{minipage}{0.225\linewidth}
			\centering
			\includegraphics[width=\linewidth]{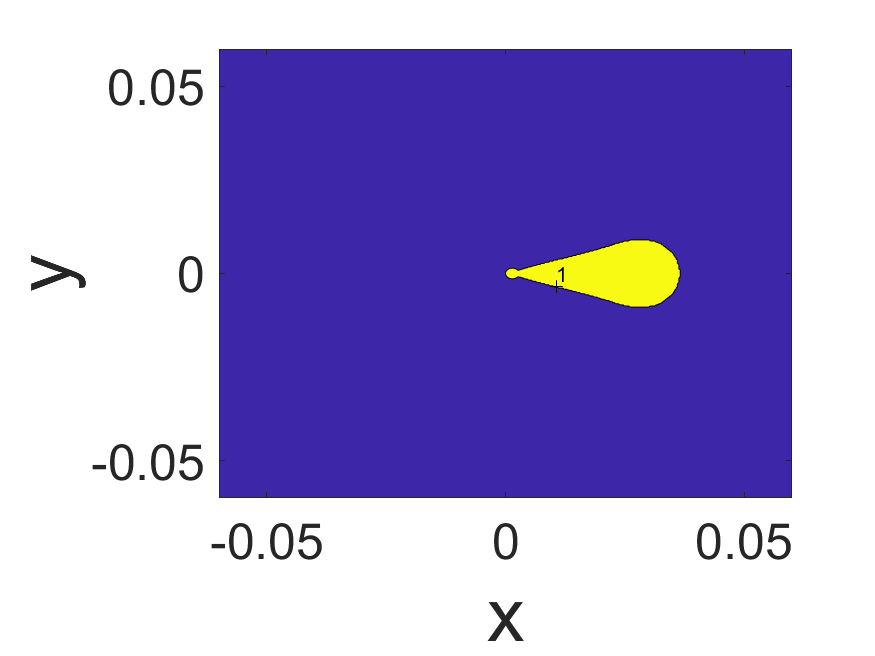}
		\end{minipage}}
  \subfigure[$\alpha = 0.4$\label{Im4}]
	{
		\begin{minipage}{0.225\linewidth}
			\centering    
			\includegraphics[width=\linewidth]{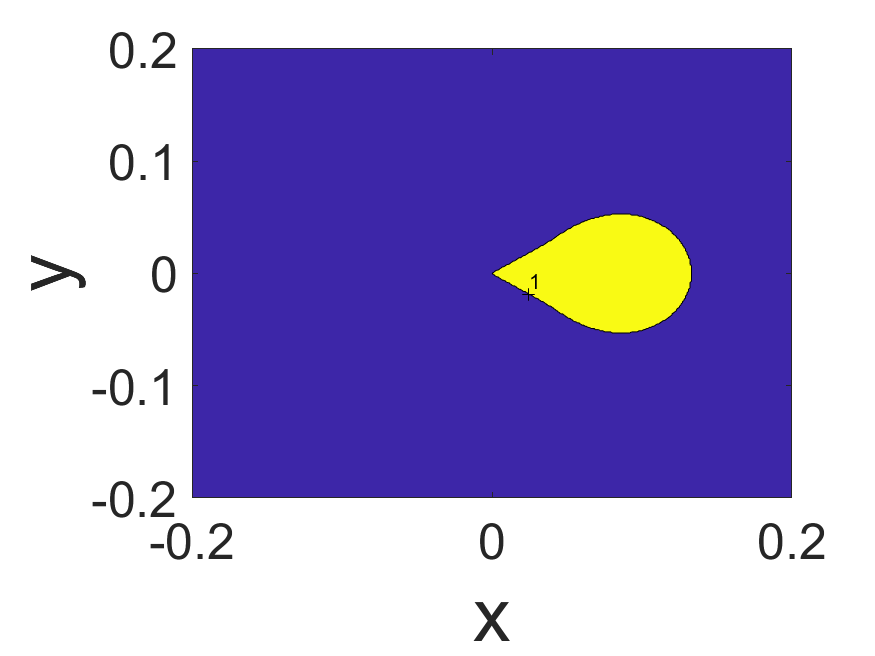} 
		\end{minipage}
	}
 \subfigure[$\alpha = 0.6$\label{Im6}] 
	{
		\begin{minipage}{0.225\linewidth}
			\centering
			\includegraphics[width=\linewidth]{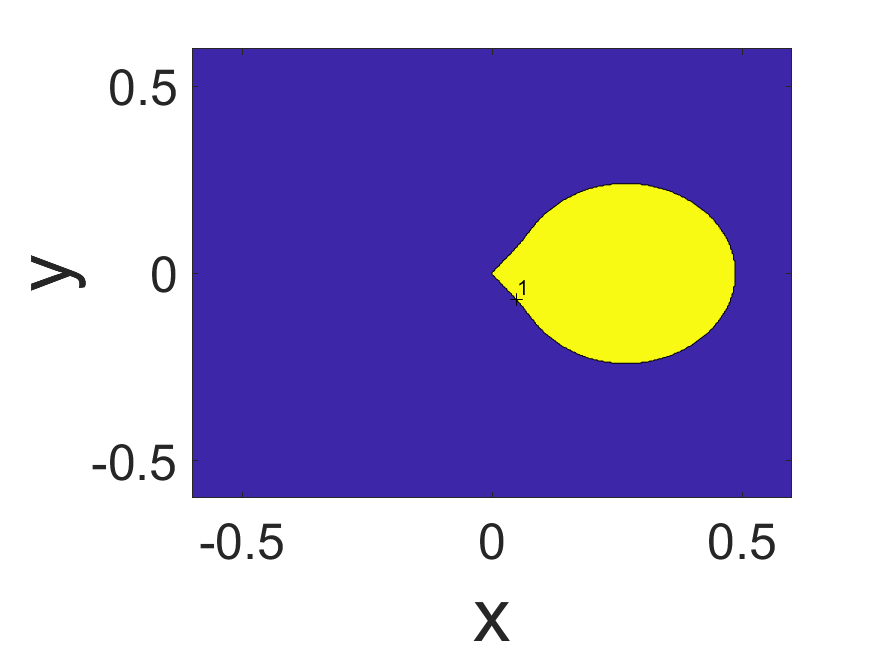}
		\end{minipage}
	}\subfigure[$\alpha = 0.8$\label{Im8}] 
	{
		\begin{minipage}{0.225\linewidth}
			\centering
			\includegraphics[width=\linewidth]{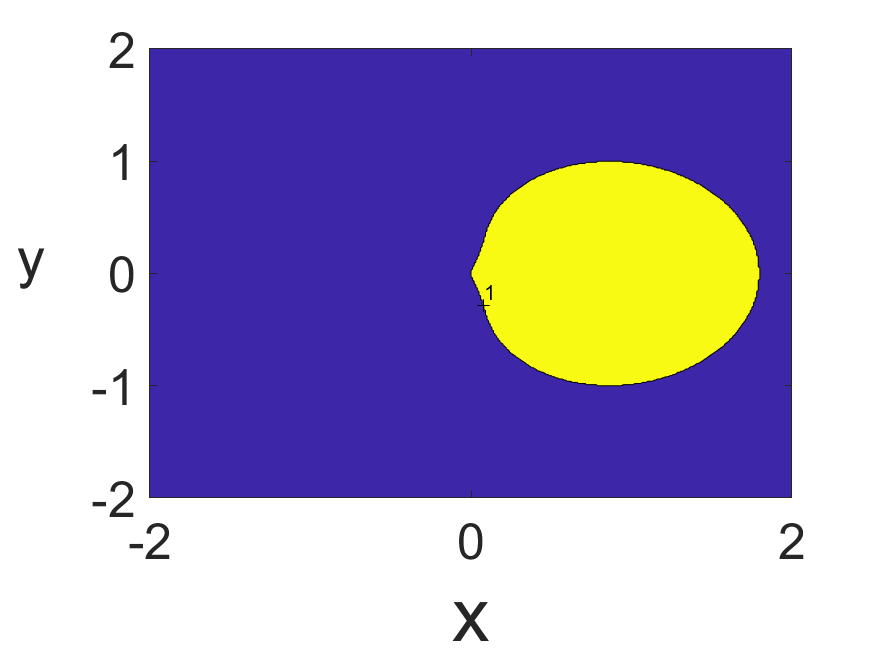}
		\end{minipage}
	}
  \caption{Contours of $\rho(\sigma)$ for the BDF-$3$ scheme  \eqref{fully scheme} for different values of fractional order $\alpha$. The blue region is the ``stable region".} \label{Implicit contours}
\end{figure}
\subsection{Accuracy test and long time simulation}\label{sec:5}
In this subsection, we present several numerical results to demonstrate the effectiveness of the proposed scheme \eqref{fully scheme} for both linear and nonlinear cases. 
\subsubsection{Linear case}
\begin{example}\label{example1}
We first consider the linear problem. In particular, we consider 
    \begin{equation}\label{pb1}
    \frac{d^{\alpha}\phi}{d t^{\alpha}} = -\lambda\phi(t) + f(t),\quad t\in[0,T],\quad \phi(0) = \phi_{0}.
\end{equation}
We shall consider the following three cases:
\begin{itemize}
\item Case I: $f(t) = \Gamma(1+\alpha)$, $\lambda = 0$. In this case, the exact solution is given by $\phi(t) = t^\alpha$.
\item Case II: $f(t) = 0,\; \lambda =1$. In this case, the exact solution is given by 
\begin{equation*}
    \phi(t) = E_{\alpha}(-\lambda t^{\alpha}),
\end{equation*}
where $E_{\alpha}(t)$ is the Mittag-Leffler function defined by  $E_{\alpha}(t) = \sum\limits_{m = 0}^{\infty}\frac{t^{m}}{\Gamma(\alpha m+1)}$.
\item Case III: $f(t) = \sin(t)$, $\lambda = 1$, and the initial value is $\phi_{0}= 1$.
\end{itemize}
\end{example}
We first show the accuracy of the spectral approximation for the $\theta$ direction (i.e., the parametric space). 
Here we set $\Delta t$ to be small enough, and we point out that for case III, we use the numerical solutions obtained by using small enough time steps as the reference solution.
We show in Figure \ref{Fig:ex1:theta} the accuracy with respect to $M$ in semi-log scale for the above three cases with different values of the fractional order $\alpha = 0.2,\, 0.8$ at time $T = 1, 20$. 

Observe that we obtain \emph{exponential convergence} for all tests, which is in agreement with the theoretical result (i.e., Theorem \ref{Theorem:fully error estimate}). This also verify the regularity result \eqref{regularity}, namely, the solution possesses a sufficient high regularity with respect to $\theta$ to deliver the spectral accuracy. It indicates that we can fix \emph{a small number} of nodes (i.e., a small value of $M$) for the spectral discretization. Therefore, we only need to solve \emph{a few number of independent integer-order ODEs} by using the characteristic decomposition (see Algorithm \ref{fast algorithm}).

\begin{figure}[htbp]
	\centering
	{
		\begin{minipage}{0.45\linewidth}
			\centering
			\includegraphics[width=\linewidth]{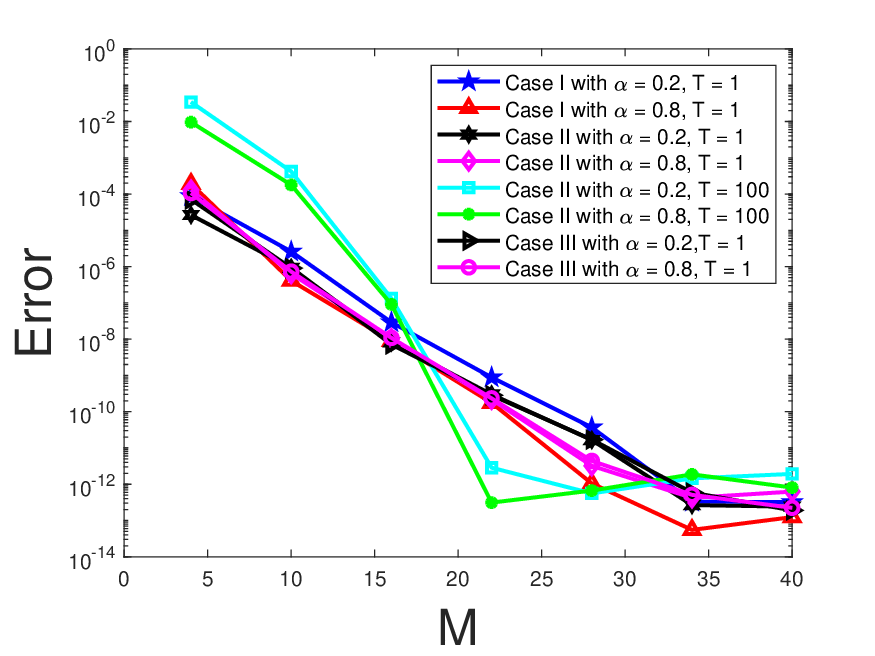}
		\end{minipage}
	}
 \caption{Example \ref{example1}. Convergence results for the spectral approximation of the extended $\theta$ direction for $\alpha = 0.2,\, 0.8$ and different times $T$. Here we set $\Delta t$ to be small enough. We obtain spectral accuracy for all test cases.} 
	 \label{Fig:ex1:theta}
\end{figure}

\begin{remark}
We remark that for the resulting $M$ independent integer-order ODEs, we can employ any efficient and high-order time discretization scheme. For instance, the BDF-$k$ schemes used in this paper or the spectral deferred method used in \cite{Chen2022}.
Consequently, by applying the BDF-$k$ scheme and fixing the value of $M$, we have that \emph{the computational cost and the storage of the fully discretization scheme are $O(N)$ and $O(1)$, respectively}, where $N$ is the number of time steps.
Additionally, we address the long-time simulation issue. 
\end{remark}

Now, let us show the convergence of the BDF-$k$ schemes. In particular, we consider the BDF-$k$ schemes with $k = 3,4,5$. 
To this end, encouraged by the previous observation, we set $M = 30$ for the spectral approximation, which is large enough to neglect the error of the spectral approximation. 

The convergence results for the three cases and different values of fractional order $\alpha$ at time $T = 1$ or $T = 20$ are shown in Figure \ref{Fig:ex1:time}, showing that the expected convergence orders are obtained for all the tests. 
The convergence result also holds for the long-time simulation; see Figure \ref{1f} for case II with $T = 20.$ 
These convergence results agree with Theorem \ref{Theorem:fully error estimate}.

\begin{remark}
We point out here that it is hard to design a unified high-order scheme with low computational cost and storage for all the cases considered above. For example, in \cite{ChengShenWang}, the spectral accuracy is obtained for case I, but fails for cases II and III. 
The fast convolution type method (see \cite{Lubich2002, Lubich2006, Jiang2017}) resolved the issues of computational cost and storage, however, it is not conveniently extended to high-order schemes.
However, we obtain in our work high-order schemes for all these cases at low computational cost and storage, and it is straightforward to design other ``good" schemes for the present EPDE.
\end{remark}

\begin{figure}[htbp]
	\centering
 \subfigure[Case I: $T = 1$\label{1b}] 
	{
		\begin{minipage}{0.45\linewidth}
			\centering
			\includegraphics[width=\linewidth]{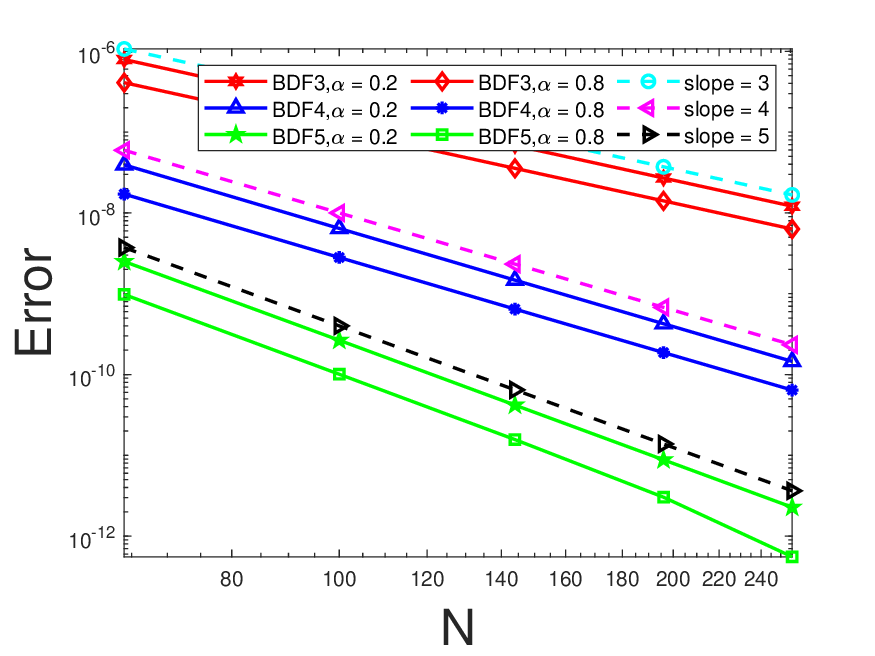}
		\end{minipage}
	}
	\subfigure[Case II: $T = 1$\label{1c}]
	{
		\begin{minipage}{0.45\linewidth}
			\centering    
			\includegraphics[width=\linewidth]{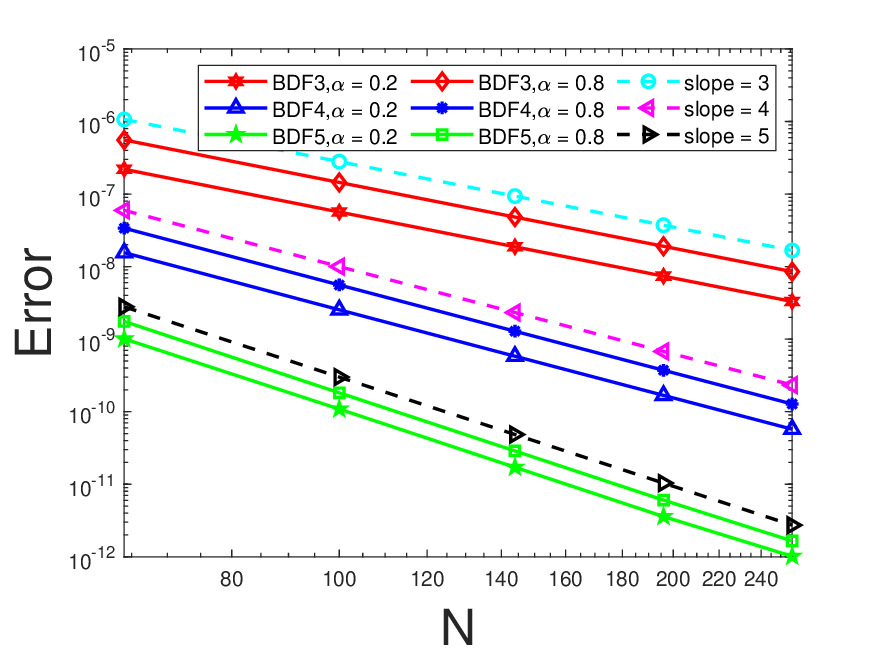} 
		\end{minipage}
	}

 \subfigure[Case II: $T = 20 $ \label{1f}] 
	{
		\begin{minipage}{0.45\linewidth}
			\centering
			\includegraphics[width=\linewidth]{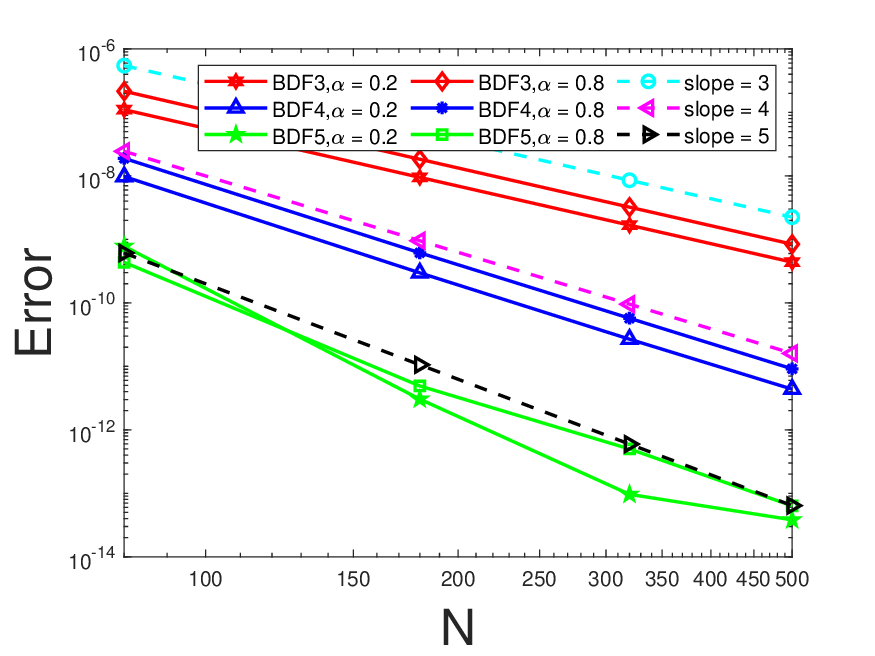}
		\end{minipage}
  }
  \subfigure[Case III: $T = 1$ \label{1g}] 
	{
		\begin{minipage}{0.45\linewidth}
			\centering
			\includegraphics[width=\linewidth]{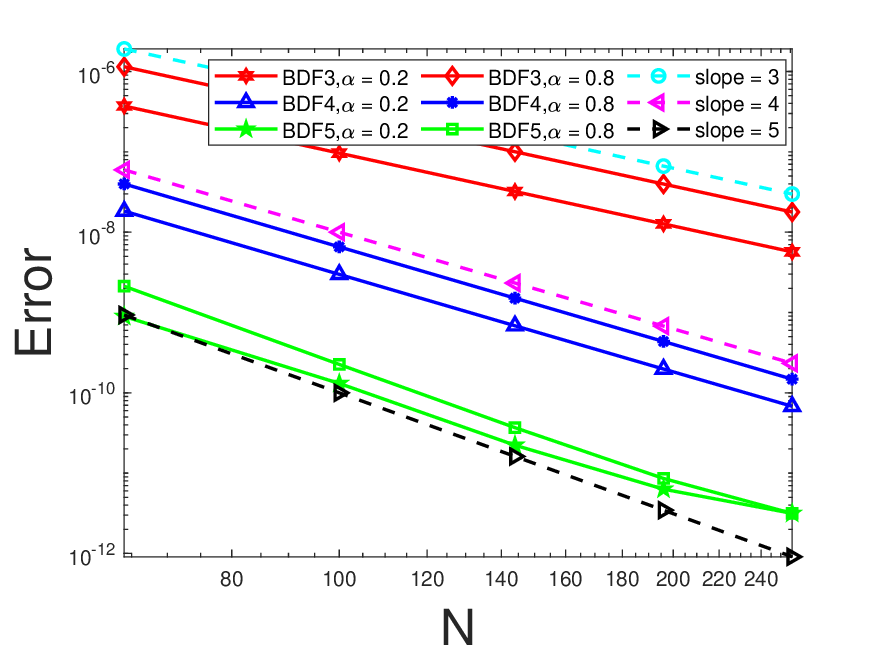}
		\end{minipage}
  }

 \caption{Example \ref{example1}. Convergence results for the BDF-$k$ ($k = 3,4, 5$) schemes for different values of fractional order $\alpha = 0.2,\, 0.8$ at $T=1$ or $T=20$. Here we set $M = 30$. The results are consistent with Theorem \ref{Theorem:fully error estimate}.} 
	 \label{Fig:ex1:time}
\end{figure}
\begin{example}\label{example2}
We now consider the nonlinear case with 
\begin{equation}
    F(t,\phi(t)) = -\lambda \phi^{3}(t) + f(t).
\end{equation}
We consider the following two cases:
\begin{itemize}
    \item Case IV: $\lambda = 1$, $f(t) = {\Gamma(3+\alpha)}t^{2}/2 + \lambda t^{6 + 3\alpha}$, and the exact solution is taken as $\phi(t) = t^{2+\alpha}$.
    \item Case V: $f(t) = 0$,
    $\lambda = 1$, and the initial value is $\phi_{0} = 1$.
\end{itemize}
\end{example}

Similarly to Example \ref{example1}, we first show the accuracy for the spectral approximation in the $\theta$ direction. To this end, we set $\Delta t$ to be small enough. For Case V, we use the numerical solution obtained by using a small enough time step $\Delta t$ and $M=50$ as the reference solution. 
For the nonlinear term, we use the Picard iteration and set the tolerance and the max iteration number to be $10^{-15}$ and $100$, respectively. 
The errors versus $M$ in semi-log scale for both cases with different values of $\alpha = 0.3,0.7$ and $T=1,100$ are shown in Figure \ref{4a}. 

We observe again that the spectral accuracy is obtained for all tests,
indicating again that only a small fixed value of $M$ is needed for the discretization of the extended parametric space. 

Consequently, we shall solve $M$ independent integer systems for each Picard iteration. 
The convergence results for the BDF-$k$ ($k = 3,4,5$) schemes with different values of $\alpha = 0.3,0.7$ and $T=1,100$ are shown in Figure \ref{4b}-\ref{4d}. Here we set $M = 30$.

We observe expected convergence order (says third, forth and fifth convergence order) for the BDF-$k$ ($k=3,4,5$) schemes for Case IV (see Figure \ref{4b}).
Unfortunately, this fails for Case V, for which we only have the convergence rate of about second order and third order for $\alpha = 0.3$ and $\alpha = 0.7$, respectively, (see Figure \ref{4c} and \ref{4d}). This is due to the lack of regularity with respect to $t$. We shall discuss and resolve this issue in a future work.

\begin{figure}[htbp]
	\centering
\subfigure[Case IV and V, error vs $M$ \label{4a}]
{
		\begin{minipage}{0.45\linewidth}
			\centering    
			\includegraphics[width=\linewidth]{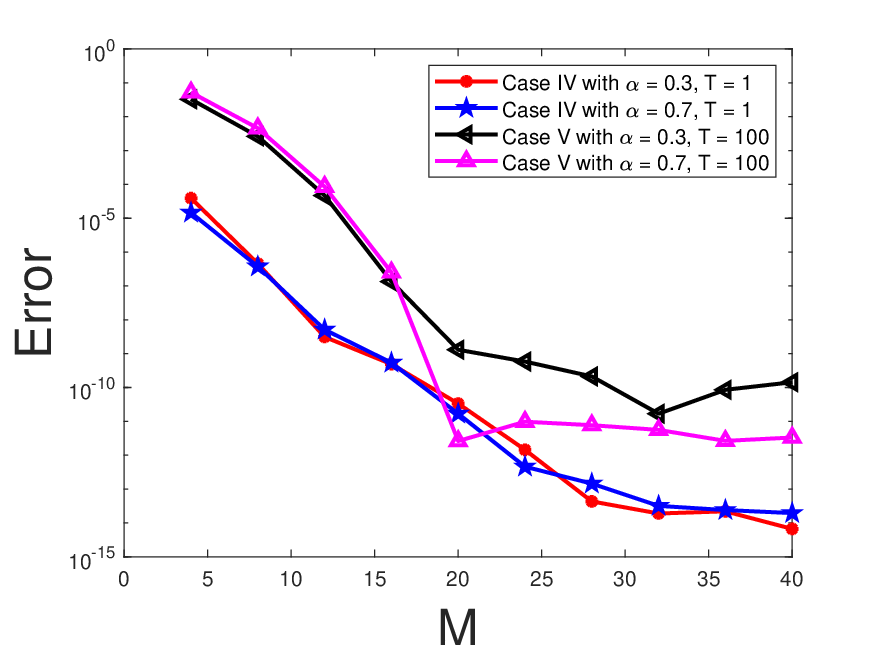} 
		\end{minipage}
	}
 \subfigure[Case IV: $ T = 1$  \label{4b}] 
	{
		\begin{minipage}{0.45\linewidth}
			\centering
			\includegraphics[width=\linewidth]{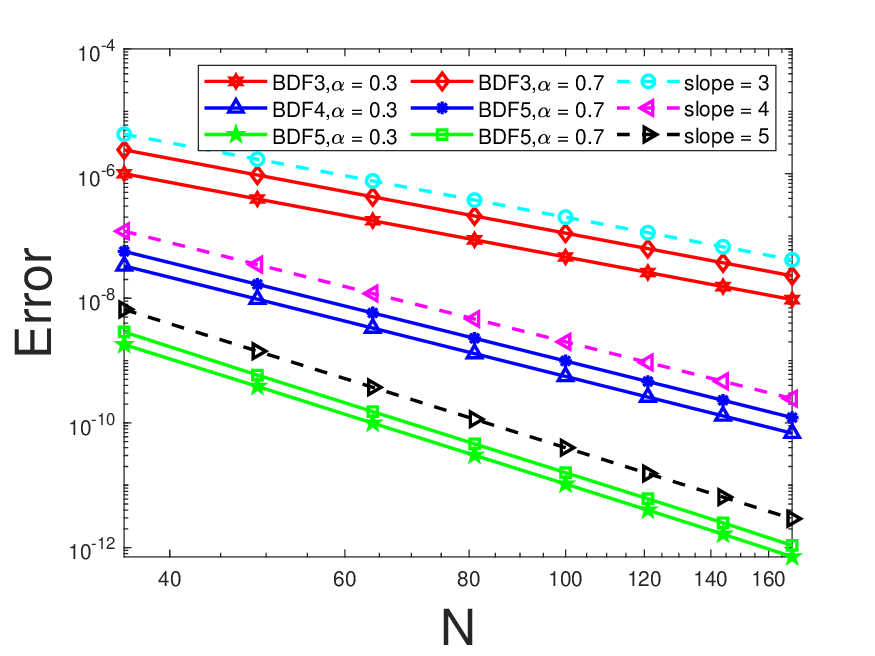}
		\end{minipage}
	}

 \subfigure[Case V: $\alpha = 0.3$  \label{4c}] 
	{
		\begin{minipage}{0.45\linewidth}
			\centering
			\includegraphics[width=\linewidth]{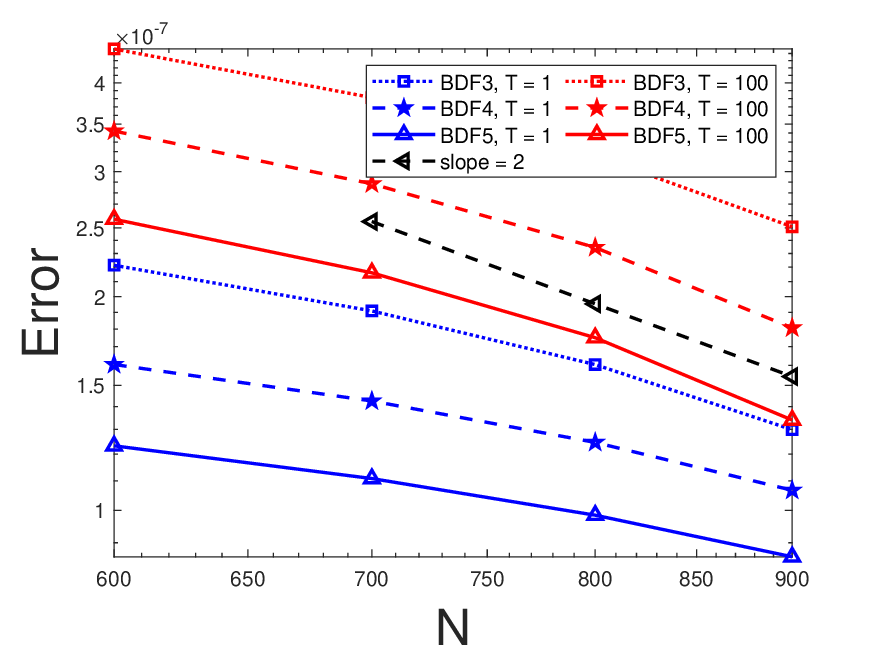}
		\end{minipage}
  }
  \subfigure[Case V: $\alpha = 0.7$  \label{4d}] 
	{
		\begin{minipage}{0.45\linewidth}
			\centering
			\includegraphics[width=\linewidth]{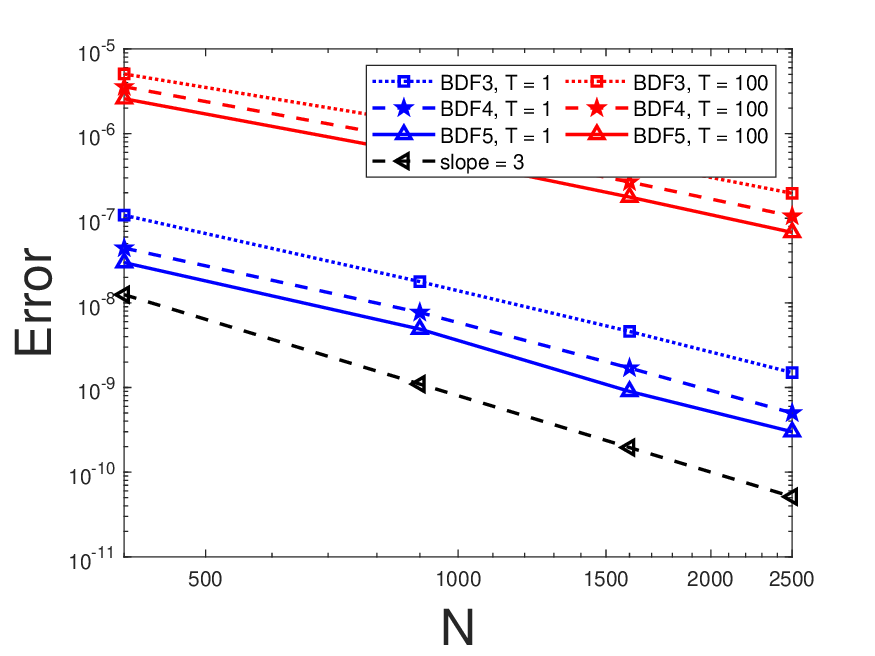}
		\end{minipage}
  }
	\caption{Example \ref{example2}. (a): Convergence results for the spectral approximation of the extended $\theta$ direction for $\alpha = 0.3,\, 0.7$ and $T = 1,100$. (b)-(d): Convergence results for the BDF-$k$ ($k = 3,4, 5$) schemes for  $\alpha = 0.3,\, 0.7$ and $T =1, 100$, here we set $M = 30$.}
 \label{Fig:ex2:time}
\end{figure}

\section{Conclusion}\label{sec:5}
In this paper, we propose a novel efficient and arbitrary high-order numerical scheme for TFDEs. We show the equivalence between TFDEs and the integer-order \emph{Extended Parametric Differential Equation} (EPDE). We  establish the stability of EPDE and show that it exhibits high regularity with respect to the extended direction $\theta$. Furthermore, we demonstrate the high regularity by employing the Jacobi spectral collocation method, showing that the spectral accuracy is obtained. 
Therefore, we could use a small number of collocation nodes to obtain a sufficiently high accuracy for the $\theta$-direction so that we could neglect the error generated by the spectral collocation method.
Consequently, by using the characteristic decomposition, we obtain $M$ (where $M$ is the number of collocation nodes) independent integer-order ODEs, which can be solved by any traditional schemes. 
In this work, we apply the BDF-$k$ scheme for the resulting $M$ independent ODEs. 
It turns out by fixing the value of $M$ that the computational cost and storage requirement of our proposed algorithm are essentially the same as those for ODEs, namely, the computational cost and the storage requirement are $O(N)$ and $O(1)$, respectively. In addition, we show the stability of the proposed schemes and give rigorous error estimates. Several numerical tests are presented showing that the expected convergence results are obtained, and the numerical results are consistent with the theoretical results.
  
It is also worth mentioning that the ideas presented in this paper can be extended to many other cases of non-local operators, such as fractional Laplace operators, non-local operators and tempered fractional order operators and so on. In future, we shall consider FDEs with general kernel and fractional integral equations. Meanwhile, we shall apply the proposed method to solve fractional PDEs, such as fractional sub-diffusion equations, fractional phase-field models \cite{Mark2017,Mark2019,Tang_PF}, etc. 
\section*{Appendix: Proof of Theorem \ref{thm:regularity}}

We give the proof of Theorem \ref{thm:regularity}, namely, we show the regularity result with respect to $\theta$ in this appendix.

\begin{proof}
We shall prove the regularity result by induction. 
Consider EPDE \eqref{equation} with $F(t,\mathcal{C}[\varphi](t)) = -\lambda\mathcal{C}[\varphi](t) + f(t)$. We then obtain from  \eqref{eq:sta:cont} that
\begin{equation*}\label{regularity1}
    \varphi\in L^{2}\left(0,T;L_{\omega^{1-\alpha,-\alpha}}^{2}(\Omega)\right).
\end{equation*}
Moreover, by multiply both sides of \eqref{equation} with $1-\theta$, taking the derivative with respect to $t$, and  letting $\widetilde{\varphi} = \frac{\partial \varphi}{\partial t}$, we obtain 
\begin{equation}\label{variable_transform}
    (1-\theta)\frac{\partial \widetilde{\varphi}}{\partial t} + \theta\widetilde{\varphi} + \lambda\mathcal{C}[\widetilde{\varphi}](t) = f'(t). 
\end{equation}
Since $f'(t)\in L^{\infty}(0,T)$, similar, we have  \eqref{eq:sta:cont}
\begin{equation*}\label{regularity2}
    \widetilde{\varphi}\in L^{2}\left(0,T;L_{\omega^{1-\alpha,-\alpha}}^{2}(\Omega)\right),\; 
    \text{i.e., } 
     \frac{\partial \varphi}{\partial t}\in L^{2}\left(0,T;L_{\omega^{1-\alpha,-\alpha}}^{2}(\Omega)\right).
\end{equation*}

Now we assume for EPDE \eqref{equation} that 
\begin{equation}\label{eq:induction:ass1}
    \frac{\partial^{m}\varphi}{\partial \theta^{m}}, \frac{\partial^{m+1}\varphi}{\partial t\partial \theta^{m}}\in L^{2}\left(0,T;L_{\omega^{1-\alpha+m,-\alpha+m}}^{2}(\Omega)\right).
\end{equation}
Next, we should show 
\begin{equation*}
    \frac{\partial^{m+1}\varphi}{\partial \theta^{m+1}}, \frac{\partial^{m+2}\varphi}{\partial t\partial \theta^{m+1}}\in L^{2}\left(0,T;L_{\omega^{2-\alpha+m,1-\alpha+m}}^{2}(\Omega)\right).
\end{equation*}

By multiply $1-\theta$ on both sides of \eqref{equation} and taking the $m+1$-th derivative with respect to $\theta$, we obtain 
    \begin{equation*}\label{regularity:01}
        \begin{aligned}
            (1-\theta)\frac{\partial^{m+2} \varphi}{\partial t\partial \theta^{m+1}} + \theta\frac{\partial^{m+1} \varphi}{\partial \theta^{m+1}} = (1+m)\left(\frac{\partial^{m+1} \varphi}{\partial t\partial \theta^{m}} - \frac{\partial^{m} \varphi}{\partial \theta^{m}}\right).
        \end{aligned}
    \end{equation*}
Multiplying $\theta\frac{\partial^{m+1}\varphi}{\partial \theta^{m+1}}$ to both sides of the above equation and integrating with respect to $\theta$ from 0 to 1 and using Cauchy-Schwarz and Young inequalities gives 
\begin{equation*}\label{regularity:02}
\begin{aligned}
     &\frac{1}{2}\frac{d}{dt}\int_{0}^{1}\left(\frac{\partial^{m+1}\varphi}{\partial \theta^{m+1}}\right)^{2}\omega^{2-\alpha+m,1-\alpha+m}d\theta + \int_{0}^{1}\left(\theta\frac{\partial^{m+1}\varphi}{\partial \theta^{m+1}}\right)^{2}\omega^{1-\alpha+m,-\alpha+m}d\theta\\
    = & (m+1)\int_{0}^{1}\left(\frac{\partial^{m+1} \varphi}{\partial t\partial \theta^{m}}  - \frac{\partial^{m} \varphi}{\partial \theta^{m}}\right)\frac{\partial^{m+1} \varphi}{\partial \theta^{m+1}}\theta\omega^{1-\alpha+m,-\alpha+m}d\theta\\
   \le  & \int_{0}^{1} \left\{ \frac{(m+1)^{2}}{4} \left[\left(\frac{\partial^{m+1} \varphi}{\partial t\partial \theta^{m}}\right)^{2} + \left(\frac{\partial^{m} \varphi}{\partial \theta^{m}}\right)^{2}\right] + \left(\theta\frac{\partial^{m+1} \varphi}{\partial \theta^{m+1}}\right)^{2} \right\} \omega^{1-\alpha+m,-\alpha+m}d\theta.
\end{aligned}
\end{equation*}
If $\frac{\partial^{m+1}\varphi}{\partial \theta^{m+1}}(0,\theta)\in L_{\omega^{2-\alpha+m,1-\alpha+m}}^{2} (\Omega)$, we deduce from the above equation and \eqref{eq:induction:ass1} that $\frac{\partial^{m+1}\varphi}{\partial \theta^{m+1}}\in L^{\infty}\left(0,T;L_{\omega^{2-\alpha+m,1-\alpha+m}}^{2}(\Omega)\right).$ Consequently, 
\begin{equation*}
    \frac{\partial^{m+1}\varphi}{\partial \theta^{m+1}}\in L^{2}\left(0,T;L_{\omega^{2-\alpha+m,1-\alpha+m}}^{2}(\Omega)\right).
\end{equation*}

We are left to show $\frac{\partial ^{m+2} \varphi}{\partial t\partial \theta^{m+1}}\in L^{2}\left(0,T;L_{\omega^{2-\alpha+m,-\alpha+m+1}}^{2}(\Omega)\right)$. 
Note that if we apply the induction to \eqref{variable_transform}, then by the assumption, we have 
\begin{equation*}
    \frac{\partial^{m}\widetilde{\varphi}}{\partial \theta^{m}}, \frac{\partial^{m+1}\widetilde{\varphi}}{\partial t\partial \theta^{m}}\in L^{2}\left(0,T;L_{\omega^{1-\alpha+m,-\alpha+m}}^{2}(\Omega)\right).
\end{equation*}
and by using the same procedure, we have from  \eqref{variable_transform} that 
\begin{equation*}
    \frac{\partial ^{m+1}\widetilde{\varphi}}{\partial \theta^{m+1}}\in L^{2}\left(0,T;L_{\omega^{2-\alpha+m,1-\alpha+m}}^{2}(\Omega)\right)\, \text{i.e., } \frac{\partial ^{m+2} \varphi}{\partial t\partial \theta^{m+1}}\in L^{2}\left(0,T;L_{\omega^{2-\alpha+m,1-\alpha+m}}^{2}(\Omega)\right).
\end{equation*}
This is the complete proof.
\end{proof}


\section*{Statements and Declarations}
\begin{itemize}
    \item Competing Interests\\
    The authors have no relevant financial or non-financial interests to disclose.
    \item Data Availability\\
    The datasets generated during and/or analysed during the current study are available from the corresponding author on reasonable request.
\end{itemize}



\end{document}